\documentclass[smallextended,nospthms]{svjour3}       
\usepackage{graphicx}
\usepackage{amsfonts}
\usepackage{amsthm}
\usepackage{amsmath}
\usepackage{amssymb}
\usepackage{times}
\usepackage{mathtools}
\usepackage{enumitem}
\usepackage{bbm}
\usepackage{xparse}
\usepackage{tikz}
\usepackage{pstricks}
\usepackage{pst-all}
\usepackage{verbatim}
\usepackage[colorlinks]{hyperref}
\usepackage[numeric,initials,nobysame,msc-links,abbrev]{amsrefs}%

\renewcommand{\eprint}[1]{\href{https://arxiv.org/abs/#1}{arXiv:#1}}
\newcommand{\pageafter}[1]{#1~pp.}
\BibSpec{article}{%
+{} {\PrintAuthors} {author}
+{,} { \textit} {title}
+{.} { } {part}
+{:} { \textit} {subtitle}
+{,} { \PrintContributions} {contribution}
+{.} { \PrintPartials} {partial}
+{,} { } {journal}
+{} { \textbf} {volume}
+{} { \PrintDatePV} {date}
+{,} { \issuetext} {number}
+{,} { \pageafter} {pages}
+{,} { } {status}
+{,} { \PrintDOI} {doi}
+{,} { available at \eprint} {eprint}
+{} { \parenthesize} {language}
+{} { \PrintTranslation} {translation}
+{;} { \PrintReprint} {reprint}
+{.} { } {note}
+{.} {} {transition}
+{} {\SentenceSpace \PrintReviews} {review}
}
\BibSpec{collection.article}{%
+{} {\PrintAuthors} {author}
+{,} { \textit} {title}
+{.} { } {part}
+{:} { \textit} {subtitle}
+{,} { \PrintContributions} {contribution}
+{,} { \PrintConference} {conference}
+{} {\PrintBook} {book}
+{,} { } {booktitle}
+{,} { \PrintDateB} {date}
+{,} { \pageafter} {pages}
+{,} { } {status}
+{,} { \PrintDOI} {doi}
+{,} { available at \eprint} {eprint}
+{} { \parenthesize} {language}
+{} { \PrintTranslation} {translation}
+{;} { \PrintReprint} {reprint}
+{.} { } {note}
+{.} {} {transition}
+{} {\SentenceSpace \PrintReviews} {review}
}

\newtheorem{thm}{Theorem}
\newtheorem{lem}[thm]{Lemma}
\newtheorem{cor}[thm]{Corollary}
\newtheorem{prop}[thm]{Proposition}

\newtheorem{clm}[thm]{Claim}
\newtheorem{fact}[thm]{Fact}
\newtheorem{defn}[thm]{Definition}
\newtheorem{rem}[thm]{Remark}

\newtheorem{obs}[thm]{Observation}
\newtheorem{ass*}{Assumption}
\numberwithin{thm}{section}

\numberwithin{equation}{section}

\setlist[itemize]{leftmargin=*}
\setlist[enumerate]{leftmargin=*,label=(\arabic*),ref=(\arabic*)}

\usetikzlibrary{arrows}
\usetikzlibrary[patterns]
\usetikzlibrary{decorations.pathreplacing}
\usetikzlibrary{shapes.misc}
\usetikzlibrary{arrows.meta}

\newcommand{\cA}{\ensuremath{\mathcal A}}
\newcommand{\cB}{\ensuremath{\mathcal B}}
\newcommand{\cC}{\ensuremath{\mathcal C}}
\newcommand{\cD}{\ensuremath{\mathcal D}}
\newcommand{\cE}{\ensuremath{\mathcal E}}
\newcommand{\cF}{\ensuremath{\mathcal F}}
\newcommand{\cG}{\ensuremath{\mathcal G}}
\newcommand{\cH}{\ensuremath{\mathcal H}}
\newcommand{\cI}{\ensuremath{\mathcal I}}

\newcommand{\cK}{\ensuremath{\mathcal K}}
\newcommand{\cL}{\ensuremath{\mathcal L}}
\newcommand{\cM}{\ensuremath{\mathcal M}}
\newcommand{\cN}{\ensuremath{\mathcal N}}

\newcommand{\cS}{\ensuremath{\mathcal S}}
\newcommand{\cT}{\ensuremath{\mathcal T}}
\newcommand{\cU}{\ensuremath{\mathcal U}}

\newcommand{\bbE}{{\ensuremath{\mathbb E}} }

\newcommand{\bbN}{{\ensuremath{\mathbb N}} }

\newcommand{\bbP}{{\ensuremath{\mathbb P}} }

\newcommand{\bbR}{{\ensuremath{\mathbb R}} }

\newcommand{\bbT}{{\ensuremath{\mathbb T}} }

\newcommand{\bbX}{{\ensuremath{\mathbb X}} }

\newcommand{\bbZ}{{\ensuremath{\mathbb Z}} }

\newcommand{\bone}{{\ensuremath{\mathbf 1}} }

    \let\d=\delta  \let\e=\varepsilon
\let\f=\varphi \let\g=\gamma \let\h=\eta    \let\k=\kappa  \let\l=\lambda
\let\m=\mu      \let\o=\omega    \let\p=\pi  
\let\r=\rho  \let\s=\sigma \let\t=\tau   
 \let\x=\xi 
\let\D=\Delta     \let\L=\Lambda 
\let\O=\Omega

\newcommand{\tc}{\thinspace |\thinspace}
\newcommand{\var}{\operatorname{Var}}

\newcommand{\1}{{\ensuremath{\mathbbm{1}}} }

\newcommand{\trelt}{\widetilde T_{\mathrm{rel}}}
\newcommand{\tbp}{{\ensuremath{\t_0^{\mathrm{BP}}}} }
\newcommand{\SG}{\cS\cG}

\newcommand{\Da}{\cD_{\mathrm{aux}}}

\newcommand{\Ta}{T_{\mathrm{aux}}}
\newcommand{\rd}{\r_{\mathrm{D}}}
\newcommand{\Ld}{L_{\mathrm{D}}}

\DeclareMathOperator{\Dom}{Dom}
\DeclareDocumentCommand \tmix { o } {%
  \IfNoValueTF {#1} {{\ensuremath{T_{\mathrm{mix}}}} }{{\ensuremath{T_{\mathrm{mix}}^{\mathrm{\scriptstyle{#1}}}}}}%
}
\DeclareDocumentCommand \trel { o } {%
  \IfNoValueTF {#1} {{\ensuremath{T_{\mathrm{rel}}}} }{{\ensuremath{T_{\mathrm{rel}}^{\mathrm{\scriptstyle{#1}}}}}}%
}

\renewcommand{\leq}{\leqslant}
\renewcommand{\geq}{\geqslant}
\renewcommand{\le}{\leqslant}
\renewcommand{\ge}{\geqslant}
\renewcommand{\to}{\rightarrow}

\journalname{Probability Theory and Related Fields}
\begin{document}

\title{Sharp threshold for the FA-2f kinetically constrained model
\thanks{This work is supported by ERC Starting Grant 680275 ``MALIG'', ANR-15-CE40-0020-01 and PRIN 20155PAWZB ``Large Scale Random Structures.''}
}
\titlerunning{FA-2f kinetically constrained model}

\author{Ivailo Hartarsky
\and 
Fabio Martinelli
\and
Cristina Toninelli}

\authorrunning{I. Hartarsky, F. Martinelli and C. Toninelli} 

\institute{I. Hartarsky, Corresponding author \at CEREMADE, CNRS, UMR 7534, Universit\'e Paris-Dauphine, PSL University, Place du Mar\'echal de Lattre de Tassigny, 75016 Paris, France\\ \email{ivailo.hartarsky@tuwien.ac.at}
           \and
           F. Martinelli \at Dipartimento di Matematica e Fisica, Universit\`a Roma Tre, Largo S.L. Murialdo, 00146 Roma, Italy\\
           \email{martin@mat.uniroma3.it}
           \and
           C. Toninelli \at CEREMADE, CNRS, UMR 7534, Universit\'e Paris-Dauphine, PSL University, Place du Mar\'echal de Lattre de Tassigny, 75016 Paris, France\\
           \email{toninelli@ceremade.dauphine.fr}
}

\date{Received: date / Accepted: date}

\maketitle
\begin{abstract}
The Fredrickson-Andersen 2-spin facilitated model on $\bbZ^d$ (FA-2f)
is a paradigmatic interacting particle system with kinetic constraints (KCM)
featuring \emph{dynamical facilitation}, an important mechanism in condensed matter physics. 
In FA-2f a site may change its state only if at least two of its nearest neighbours are empty. Although the process is reversible w.r.t.\ a product Bernoulli measure, it is not \emph{attractive} and features degenerate jump rates and anomalous divergence of characteristic time scales as the density $q$ of empty sites tends to $0$. A natural random variable encoding the above features is $\t_0$, the first time at which the origin becomes empty for the stationary process. Our  
main result is the sharp threshold 
\[\t_0=\exp\Big(\frac{d\cdot\l(d,2)+o(1)}{q^{1/(d-1)}}\Big)\quad \text{w.h.p.}\]
with $\l(d,2)$ the sharp threshold constant for 2-neighbour bootstrap percolation on $\bbZ^d$, the monotone deterministic automaton counterpart of FA-2f.
This is the first sharp result for a critical KCM
and it compares with Holroyd's 2003 result on bootstrap percolation and its subsequent improvements. It also settles various controversies accumulated in the physics literature over the last four decades. Furthermore, our novel techniques enable  completing the recent ambitious program on the universality phenomenon for critical KCM and establishing sharp thresholds for other two-dimensional KCM.

\end{abstract}

\section{Introduction} 
Fredrickson-Andersen $j$-spin facilitated models (FA-$j$f) are a class of interacting particle systems that were introduced by physicists in the 1980s \cite{Fredrickson84} to model the liquid/glass transition, a major and still largely open problem in condensed matter physics \cites{Berthier11,Arceri20}. Later on, several  models with different update rules were introduced, and this larger class has been dubbed  \emph{Kinetically Constrained Models} (KCM)  (see \emph{e.g.}\ \cite{Garrahan11} and references therein). The key feature of KCM is that an update at a given vertex $x$ can occur only if a suitable neighbourhood of $x$ contains only holes, the facilitating vertices. The presence of this dynamical constraint gives rise to a mechanism dubbed \emph{dynamical facilitation} \cite{Speck19} in condensed matter physics: motion on smaller scales begets motion on larger scales. Extensive numerical simulations indicate that indeed KCM can display a remarkable glassy behaviour, featuring in particular an anomalous divergence of characteristic time scales. 
As a good representative of a random variable whose law encodes the above behaviour one could take $\t_0$, the first time the origin becomes a hole (or infected, in the jargon of the sequel). In the last forty years physicists have put forward several different conjectures on the scaling of $\t_0$ as the equilibrium density of the  holes goes to zero for FA-$j$f models. However, to date  a clear cut answer on the form of this scaling has proved elusive  
due to the very slow dynamics and large finite size effects intrinsic to its glassy dynamics. 

From the mathematical point of view, the study of FA-$j$f and KCM in general poses very challenging problems. 
This is largely due to the fact that these models do \emph{not} feature an \emph{attractive dynamics} (in the sense of \cite{Liggett05}*{Chapter III}), and therefore  many of the powerful tools developed to study attractive stochastic spin dynamics, \emph{e.g.}\ monotone coupling or censoring, cannot be used. A central issue has been therefore that of developing novel mathematical tools to determine the long time behaviour of the stationary process and, more specifically, to find the scaling of  the associated infection time of the origin, $\t_0$ in the sequel, as the density $q$ of the empty sites (the facilitating ones) shrinks to zero. 

With this motivation, an ambitious program was recently initiated in \cite{Martinelli19} to determine as accurately as possible the divergence of the infection time for the stationary process, as $q\to 0$ for the FA-$j$f models in any dimension and for general KCM in two dimensions. This program mirrors in some aspects the analogous program for general $\cU$-bootstrap percolation cellular automata ($\cU$-BP) launched by \cite{Bollobas15} and carried out in \cites{Bollobas14,Balister16} and for $j$-neighbour bootstrap percolation \cites{Holroyd03,Balogh12,Gravner08,Hartarsky19}. Indeed  $\cU$-BP models and $j$-neighbour bootstrap percolation can be viewed as the monotone deterministic counterparts of generic KCM and FA-$j$f models respectively. Despite the above analogy, the lack of monotonicity for KCM induces a much more complex behaviour and richer universality classes than BP \cites{Martinelli19a,Mareche20Duarte,Hartarsky21a,Hartarsky20,Hartarsky22univlower,Hartarsky20II}. 

In spite of several important advances \cites{Cancrini08,Martinelli19,Martinelli19a,Hartarsky21a,Mareche20Duarte,Hartarsky20}, the \emph{sharp} estimates of the divergence of $\t_0$ for stationary KCM still remained a  milestone open problem. Solving it requires discovering  
the optimal infection/healing mechanism to reach the origin and crafting 
the mathematical tools to transform  
the knowledge of this mechanism into tight upper and lower bounds for $\t_0$ for the stationary process. In this paper we solve this problem (see Section \ref{subsec:heuristics} for an account of our most prominent innovations) for the first time and we establish the sharp scaling for FA-$2$f models in any dimension (Theorem \ref{th:FA2f}). In doing so, we also settle various unresolved controversies in the physics literature (see Section \ref{sec:phys} for a detailed account).

Our novel approach not only leads to deeper results, but also extends in breadth. Indeed, it opens the way for accomplishing the final step \cite{Hartarsky20II} to complete the program of \cite{Martinelli19} for establishing KCM universality.

\subsection{Bootstrap percolation background}
\label{sec:BP}
Let us start by recalling some background on  $j$-neighbour bootstrap percolation. Let $\O=\{0,1\}^{\bbZ^d}$ and call a site $x\in\bbZ^d$ \emph{infected} (or \emph{empty}) for $\o\in \O$ if $\o_x=0$ and \emph{healthy} (or \emph{filled}) otherwise. For fixed $0<q<1$, we denote by $\m_q$ the product Bernoulli probability measure with parameter $q$ under which each site is infected with probability $q$. When confusion does not arise, we write $\mu=\mu_q$. 
Given two integers $1\le j\le d$ the $j$-neighbour BP model ($j$-BP for short) on the $d$-dimensional lattice $\bbZ^d$  is the monotone cellular automaton on $\O$ evolving as follows. Let $A_0\subset\bbZ^d$ be the set of \emph{initially infected sites} distributed according to $\m$. Then for any integer \emph{time} $t\ge 0$ we recursively define
\[
A_{t+1}=A_t\cup\big\{x\in\bbZ^d,|N_x\cap A_t|\ge j\big\},
\]
where $N_x$ denotes the set of neighbours of $x$ in the usual graph structure of $\bbZ^d$. In other words, a site becomes infected \emph{forever} as soon as its constraint becomes satisfied, namely as soon as it has at least $j$ already infected neighbours. 
\begin{rem}
The $j$-BP is clearly monotone in the initial set of infection \emph{i.e.}\ $A_t\subset A'_t$ for all $t\ge 1$ if $A_0\subset A'_0$. Such a monotonicity will, however, be missing in the KCM models analysed in this work. 
\end{rem}
A key quantity  for bootstrap percolation is the \emph{infection time of the origin} defined as $\tbp=\inf\{t\ge0,0\in A_t\}$. For $j$=1, trivially,  $\tbp$ scales as the distance to the origin of the nearest infected site and thus behaves w.h.p.\ as $q^{-1/d}$. For $j>1$, the typical value of $\tbp$ w.r.t.\ $\mu_q$ has been investigated in a series of works, starting with the seminal paper of Aizenman and Lebowitz \cite{Aizenman88} and Holroyd's breakthrough \cite{Holroyd03} determining a sharp threshold for $d=j=2$. We refer to \cite{Morris17a} for an account of the field and only recall the more recent results that include second order corrections to the sharp threshold. Here and throughout the paper, when using asymptotic notation we refer to $q\to 0$.\footnote{If $f$ and $g$ are real-valued functions of $q$ with $g$ positive, we  write $f = O(g)$ if there exists a (deterministic absolute) constant $C > 0$ such that $|f(q)|\le C g(q)$ for every sufficiently small $q > 0$. We also write $f = \O(g)$ if $f$ is positive and $g = O(f)$. We further write $f = \Theta(g)$ if both $f = O(g)$ and $f = \O(g)$. Finally, we write $f = o(g)$ if for all $c>0$ for sufficiently small $q>0$ we have $|f(q)|\le c g(q)$.
}
For 2-BP in $d=2$, w.h.p.\ it holds
\cites{Hartarsky19,Gravner08} that
\begin{equation}\label{eq:2nBP}
\tbp=\exp\Big(\frac{\pi^2}{18q}\big(1-\sqrt{q}\cdot\Theta(1)\big)\Big).
\end{equation}
For $j$-BP for all $d\geq j\ge 2$, w.h.p.\ it holds
\cites{Uzzell19,Balogh12} 
\begin{align}
\label{eq:jnbp:lb}\tbp&{}\ge\exp^{j-1}\Big(\frac{\l(d,j)}{q^{1/(d-j+1)}}\big(1-o(1)\big)\Big),\\
\label{eq:jnbp:ub}
\tbp&{}\le\exp^{j-1}\Big(\frac{\l(d,j)}{q^{1/(d-j+1)}}\big(1-\O\big(q^{1/(2(d-j+1))}\big)\big)\Big),
\end{align}
where $\exp^k$ denotes the exponential iterated $k$ times and
$\l(d,j)$ are the positive constants defined explicitly in
\cite{Balogh09a}*{(1)-(3)}. We recall that $\l(2,2)=\pi^2/18$
\cite{Holroyd03}*{Proposition 5} and we refer the interested reader to \cite{Balogh09a}*{Table 1 and Proposition 4} for other values of $d,j$.

We are now ready to introduce the Fredrickson-Andersen model, a natural stochastic counterpart of $j$-BP and the main focus of this work. 

\subsection{The Fredrickson--Andersen model and main result}
\label{subsec:models}
For integers $1\le j\le d$ the \emph{Fredrickson--Andersen $j$-spin facilitated model} (FA-$j$f) is the interacting particle system on $\O=\{0,1\}^{\bbZ^d}$ constructed as follows. Each site is endowed with an independent Poisson clock with rate $1$. At each clock ring the state of the site is updated to an independent Bernoulli random variable with parameter $1-q$ subject to the crucial constraint that if the site has fewer than $j$ infected (nearest) neighbours currently, then the update is rejected. We refer to updates occurring at sites with at least $j$ infected neighbours at the time of the update as \emph{legal}.
\begin{rem}
Contrary to the $j$-BP model, the FA-$j$f process is clearly non-monotone because of the possible recovery of infected sites with at least $j$ infected neighbours. This feature is one of the major obstacles in the analysis of the process.
\end{rem} It is standard to show (see \cite{Liggett05}) that the FA-$j$f process is well defined and it is reversible w.r.t.\ $\m_q$. When the
initial distribution at time $t=0$ is a measure $\nu$, the law and expectation
of the process on the Skorokhod space $D([0,\infty) , \O)$ will be
denoted by $\bbP_{\nu}$ and $\bbE_{\nu}$ respectively. As for $j$-BP let 
\[
\t_0=\inf\{t\ge 0,\o_0(t)=0\}
\]
be the first time the origin becomes infected. Our main goal is to quantify precisely $\bbE_{\m_q}[\t_0]$, the average of $\t_0$ w.r.t.\ the stationary process as $q\to 0$. 
In order to keep the setting simple and the results more transparent, we will focus on the FA-2f model. Other models, including FA-$j$f for all values of $3\le j\le d$, are discussed in Section \ref{subsec:further:results}. Recall the constants $\l(d,2)$ from \eqref{eq:jnbp:lb}, \eqref{eq:jnbp:ub}, so that $\l(2,2)=\pi^2/18$.
\begin{thm}
\label{th:FA2f}
As $q\to 0$ the stationary FA-$2$f model on $\bbZ^d$  satisfies: 
\begin{align}
\label{eq:FA2f:lower}
\bbE_{\m_q}[\tau_0]&\ge{}\exp\Big( \frac{\pi^2}{9q}\big(1-\sqrt{q}\cdot O(1)\big)\Big),\\
\label{eq:FA2f:upper}\bbE_{\m_q}[\tau_0]&\leq{}\exp\Big( \frac{\pi^2}{9q}\big(1+\sqrt{q}\cdot\log^{O(1)}(1/q)\big)\Big),\\
\intertext{if $d=2$, and}
\label{eq:lbd}
\bbE_{\m_q}[\tau_0]&\geq\exp\Big( \frac{d\cdot\l(d,2)}{ q^{1/(d-1)}}(1-o(1))\Big),\\
\label{eq:ubd}
\bbE_{\m_q}[\tau_0]&\leq
\exp\Big( \frac{d\cdot\l(d,2)}{q^{1/(d-1)}}\big(1+q^{1/(2(d-1))}(\log(1/q))^{O(1)}\big)\Big),
\end{align}
if $d\ge 3$. Moreover, \eqref{eq:FA2f:lower}-\eqref{eq:ubd} also hold for $\t_0$ w.h.p.
\end{thm}
In particular, recalling \eqref{eq:jnbp:lb}, \eqref{eq:jnbp:ub}, we have the following.
\begin{cor}
W.h.p.\ $\t_0=(\tbp)^{d+o(1)}$.
\end{cor}
The above are the first results that establish the sharp asymptotics of $\log\bbE_{\m_q}[\t_0]$ within the whole class of ``critical'' KCM. 

\begin{rem}
We will not provide an explicit proof of the case $d\ge 3$ as it does not require any additional effort with respect to the case $d=2$.
The only significant difference is that the lower bound from \eqref{eq:2nBP} is not available in higher dimensions, leading to a corresponding weakening of the lower bound \eqref{eq:lbd} as compared to \eqref{eq:FA2f:lower}.
\end{rem}
\begin{rem} Despite the resemblance, our results are by no means a corollary of their 2-BP counterpart \eqref{eq:2nBP}. 
While the lower bounds \eqref{eq:FA2f:lower} and \eqref{eq:lbd} do indeed follow rather easily from \eqref{eq:2nBP} and \eqref{eq:jnbp:lb}
together with an improvement of the ``automatic'' lower bound from \cite{Cancrini08}*{Theorem 6.9}, the proof of \eqref{eq:FA2f:upper} and \eqref{eq:ubd} is much more involved. In particular, it  requires guessing an efficient infection/healing mechanism to infect the origin, which has no counterpart in the monotone $j$-BP dynamics (see Section \ref{subsec:heuristics}).
\end{rem}
\subsection{Extensions}
\label{subsec:further:results}

\subsubsection{FA-$j$f with $j\neq 2$}
For the sake of completeness, let us briefly discuss the FA-$j$f model with other values of $j$. The case $j=1$ is the simplest to analyse and behaves very differently: relaxation is dominated by the motion of single infected sites and time scales diverge as $1/q^{\Theta(1)}$ (see \cites{Cancrini08, Shapira20} for the values of the exponent). For $d\ge j\ge 4$ we believe that minor modifications of the treatment of \cite{Martinelli19} along the lines provided by \cite{Balogh09a} should be sufficient to prove that $\bbE_{\m_q}[\t_0]$ scales as $\tbp$ (see \eqref{eq:jnbp:lb}, \eqref{eq:jnbp:ub}). The only remaining case, $d\ge j=3$, should require some more work, still following the approach of \cite{Martinelli19}. Let us emphasise that it should be possible to treat all $d\ge j\ge 3$, using the techniques of the present paper. However, the much faster divergence of the scaling involved should allow the less refined technique of \cite{Martinelli19} to work, as there is a much larger margin for error, making those results easier. We leave the above considerations to future work.

\subsubsection{More general update rules: $\cU$-KCM} 
The full power of the method developed in the present work is required to treat two-dimensional $\cU$-KCM,  a very general class of interacting particle systems with kinetic constraints on $\bbZ^2$. These models and their bootstrap percolation counterpart, $\cU$-BP, are defined similarly to  FA-$j$f and $j$-BP but with arbitrary local monotone constraints (or update rules) $\cU$ \cites{Cancrini08,Martinelli19}. There exist several very symmetric constraints, including the so-called modified 2-BP, requiring two \emph{non-opposite} neighbours to be infected, for which the exact asymptotics of $\log\tbp$, and sometimes even the higher order corrections, are known \cites{Bringmann12}. Our methods should adapt to this setting to yield equally sharp results for $\bbE_{\m_q}[\t_0]$ of the corresponding $\cU$-KCM. In this general setting the outcome would again be of the form $\bbE_{\mu_q}[\t_0]\simeq (\tbp)^{2}$ as for FA-$2$f.

We warn the reader that the exponent $2$ in two dimensions relating $\bbE_{\mu_q}[\t_0]$ to $\tbp$ is not general \cites{Hartarsky20,Hartarsky22univlower} and only applies to `isotropic' models \cite{Hartarsky20II}. Nevertheless, developing the approach of the present work further, in \cite{Hartarsky20II} $\log \bbE_{\m_q}[\t_0]$ is determined up to a constant factor for all so-called ``critical'' KCM in two dimensions, matching the lower bounds established in \cite{Hartarsky22univlower} and establishing a richer KCM analogue of the BP universality result of \cite{Bollobas14}. 

\subsection{Settling a controversy in the physics literature}\label{sec:phys}
Soon after the FA-$j$f models were introduced, some  conjectures in the physics literature predicted the divergence of $\bbE_{\m_q}[\t_0]$ at a \emph{positive} critical density $q_c$ (\cites{Fredrickson84,Fredrickson85,Graham97}). These conjectures were subsequently ruled out in \cite{Cancrini08}, the first contribution analysing rigorously FA-$j$f. After \cite{Cancrini08} and prior to the present work, the  best known bounds on the infection time 
were
\begin{align}
\exp\Big(\frac{\O(1)}{q^{1/(d-1)}}\Big)\le{} &\bbE_{\m_q}[\t_0]\le\exp\Big(\frac{\log^{O(1)}(1/q)}{q^{1/(d-1)}}\Big),\quad j=2,\label{eq:FA2f:old}\\
\exp^{j-1}\Big(\frac{\l(d,j)-o(1)}{q^{1/(d-j+1)}}\Big)\le{} &\bbE_{\m_q}[\t_0]\le\exp^{j-1}\Big(\frac{O(1)}{q^{1/(d-j+1)}}\Big), \quad j\ge 3.\nonumber
\end{align}
The lower bounds follow from the general lower bound  \cite{Martinelli19}*{Lemma 4.3} $\mathbb E_{\m_q}[\t_0]=\Omega(\text{median of $\tbp$})$ together with the $j$-nBP lower bounds (see Section \ref{sec:BP}) while the upper bounds were recently obtained by the second and third author in \cite{Martinelli19}. As such, the above results do not settle a controversy between several conjectures that were put forward in the physics literature. 

The first quantitative prediction for the scaling of $E_{\m_q}[\t_0]$ appeared in \cite{Nakanishi86} where, based on numerical simulations, a faster than exponential divergence in $1/q$ was conjectured for FA-$2$f in $d=2$. For the latter, the first to claim an exponential scaling $\exp(\Theta(1)/q)$ was Reiter \cite{Reiter91}. He argued that the infection process of the origin is dominated by the motion of \emph{macro-defects}, \emph{i.e.}\ rare regions having probability
$\exp(-\Theta(1)/q)$ and size poly($1/q$) that move at an exponentially small rate $\exp(-\Theta(1)/q)$. Later Biroli, Fisher and the last author \cite{Toninelli05} considerably refined the above picture. They argued that macro-defects should coincide with the critical droplets of $2$-BP having probability $\exp(-\pi^2/(9 q))$ and that  
the time scale of the relaxation process inside a macro-defect should be $\exp(c/\sqrt q)$, \emph{i.e.}\ sub-dominant with respect to the inverse of their density, in sharp contrast with the prediction of \cite{Reiter91}. Based on this and on the idea  that macro-defects move diffusively, the relaxation time scale of FA-2f in $d=2$ was conjectured to diverge as $\exp(\pi^2/(9q))$ in $d=2$  \cite{Toninelli05}*{Section 6.3}. 
Yet, a different prediction was later made in \cite{Teomy15} implying a different scaling of the form $\exp(2\pi^2/(9q))$. 
Concerning the behaviour of FA-$2$f in higher dimensions, in  \cite{Toninelli05} 
the relaxation time was predicted to  diverge
as $(\tau_0^{\rm BP})^d$, though the prediction was less precise than for the two dimensional case since  the sharp results for $2$-BP in dimension $d>2$ proved in \cite{Balogh12} were yet to be established.  

Theorem \ref{th:FA2f} settles the above controversy by confirming  the scaling prediction of \cites{Toninelli05,Reiter91} and by disproving those of \cites{Teomy15,Nakanishi86}. Moreover, our result  on the characteristic time scale of the relaxation process \emph{inside} a macro-defect (see Proposition \ref{thm:droplet}) agrees with the prediction of \cite{Toninelli05} and disproves the one of \cite{Reiter91}.

\subsection{Behind Theorem \ref{th:FA2f}: high-level ideas}
\label{subsec:heuristics}
The main intuition behind Theorem \ref{th:FA2f} is that for $q\ll 1$  the
relaxation to equilibrium of the stationary  FA-$2$f process is dominated by
the slow motion of patches of infection dubbed
\emph{mobile droplets} or just \emph{droplets} with very small probability of occurrence, roughly $\exp(-\pi^2/(9q))$. In analogy with the \emph{critical droplets} of bootstrap
percolation (see \cite{Holroyd03}),
mobile droplets have a linear size which is polynomially increasing in $q$
(with some arbitrariness), \emph{i.e.}\
they live on a much smaller scale than the metastable length scale $e^{\Theta(1/q^{{1/(d-1)}})}$ arising in 2-BP
percolation model. One of the main requirements dictating the choice of
the scale of mobile droplets is the requirement that the typical infection
environment around a droplet is w.h.p.\ such that the droplet is able
to move under the FA-$2$f dynamics in any direction. 
Within this scenario the main contribution to the infection time of the
origin for the stationary  FA-$2$f process should come from the time it takes for a droplet to reach the origin. 

In order to translate the above intuition into a mathematically
rigorous proof, one is faced with two different fundamental problems: 
\begin{enumerate}
\item\label{item:definition} a precise, yet workable, definition of mobile droplets;
  \item\label{item:evolution} an efficient model for their ``effective'' random evolution.
\end{enumerate}
In \cites{Martinelli19,Martinelli19a,Hartarsky21a} mobile droplets (dubbed ``super-good'' regions there) have been defined rather
rigidly as fully infected regions of suitable shape and size and their motion has been modelled as a
\emph{generalised FA-$1$f process} on $\bbZ^2$ \cite{Martinelli19a}*{Section 3.1}. In the latter process mobile droplets are freely created or destroyed with the correct
heat-bath equilibrium rates but \emph{only at locations which are
adjacent to an already
existing droplet}. The main outcome of these papers have been (upper) bounds on
the infection time of the origin of the form $\t_0\le
1/\rd^{\log\log(1/\rd
)^{O(1)}}$ w.h.p., where $\rd$ is the density of mobile droplets.

While rather powerful and robust, the solution proposed in \cites{Martinelli19,Martinelli19a,Hartarsky21a} to \ref{item:definition} and \ref{item:evolution} above has no chance to get the \emph{exact}
asymptotics of the infection time because of the rigidity in the
definition of the mobile droplets \emph{and} of the chosen model for their
effective dynamics. Indeed, a mobile droplet should be allowed to
deform itself and move to a nearby position like an
amoeba, by rearranging its infection using the FA-$2$f
moves. This ``amoeba motion'' between nearby locations should occur on a time scale much smaller than the global time scale
necessary to bring a droplet from far away to the origin. In particular, it should not require to first create a new
droplet from the initial one and only later destroy the original one (the main mechanism of the droplet dynamics under the generalised FA-$1$f process).  

With this in mind we offer a new solution to \ref{item:definition} and \ref{item:evolution} above which indeed leads to determining the exact asymptotics of the infection time. Concerning \ref{item:definition}, our treatment in Section
\ref{sec:droplet} consists of two steps. We first propose a
sophisticated multiscale definition of mobile droplets which, in
particular, introduces a  crucial degree of \emph{softness} in their microscopic infection's configuration\footnote{This construction is inspired by one suggested by P.~Balister in 2017, which he conjectured would remove the spurious log-corrections in the bound \eqref{eq:FA2f:old} available at that time.}. The second and much more technically involved step is developing the tools necessary to analyse the FA-2f dynamics inside a mobile droplet. In particular, we then prove two key features (see Propositions
\ref{prop:probadroplet} and 
\ref{thm:droplet} for the case $d=2$):
\begin{enumerate}[label=(1.\alph*),ref=(1.\alph*)]
\item\label{item:rd} to the leading order the probability $\rd$ of mobile droplets satisfies \[\rd
\ge \exp{\Big(- \frac{d\lambda(d,2)}{q^{1/(d-1)}}-\frac{O(\log^2(1/q))}{q^{1/(2d-2)}}\Big)},\]
\item\label{item:gd} the ``amoeba motion'' of mobile droplets between
nearby locations occurs on a time scale $\exp(O(\log(1/q)^3)/{q^{1/(2d-2)}})$ which is sub-leading w.r.t.\ the main time scale
of the problem and only manifests in the second term of \eqref{eq:FA2f:upper}.   \end{enumerate}
Property \ref{item:rd} follows rather easily from well known facts from bootstrap percolation theory, while proving property \ref{item:gd}, one of the most innovative steps of the paper, requires a substantial amount of new ideas.

While properties \ref{item:rd} and \ref{item:gd} above are essential, they are not sufficient on their own for solving problem \ref{item:evolution} above. In Section
\ref{sec:upperbound} we propose to model
(admittedly only at the level of a Poincar\'e inequality, which
however suffices for our purposes) the
random evolution of mobile droplets as a symmetric simple exclusion process with
two additional crucial add-ons: a \emph{coalescence} part (when two mobile droplets meet
they are allowed to merge) and a \emph{branching} part (a single droplet can
create a new one nearby as in the generalised FA-$1$f process). This model, which we call $g$-CBSEP, was studied for the purpose of its present application in the preparatory work \cite{Hartarsky22CBSEP}. Finally, the fact that $g$-CBSEP  relaxes on a time scale proportional to the inverse density of mobile droplets (modulo logarithmic corrections) (see Proposition \ref{prop:g-CBSEP}) yields the scaling of Theorem \ref{th:FA2f}. We emphasise that modelling the large-scale motion of droplets by $g$-CBSEP instead of a generalised FA-1f process is an absolute novelty, also with respect to the physics literature.

\section{Proof of Theorem \ref{th:FA2f}: lower bound}
\label{sec:lower}
In this section we establish the lower bounds \eqref{eq:FA2f:lower} and \eqref{eq:lbd} of Theorem \ref{th:FA2f}. Our proof is actually a procedure to establish a general lower bound for $\bbE_{\mu_q}[\t_0]$ based on bootstrap percolation. This approach improves upon a previous general result \cite{Martinelli19}*{Lemma 4.3} 
which lower bounds $\bbE_{\mu_q}[\t_0]$ with the mean infection time for the corresponding bootstrap percolation model.

Before spelling the details out, let us explain the proof idea. In BP it is known that the origin typically gets infected by a rare ``critical droplet'' of size roughly $1/q$ which can be infected only using internal infections. This droplet, initially at distance $\approx (\text{density of critical droplets})^{-1/d}$ from the origin, grows linearly until hitting the origin. Hence $\tbp\approx (\text{density of critical droplets})^{-1/d}$. On the contrary, the leading behaviour of $\t_0$ is  governed by the inverse probability of a critical droplet, because one needs to wait for a critical droplet to reach the origin under the FA-2f dynamics. Thus, we expect $\t_0\approx(\tbp)^{d}$.

In order to turn this idea into a proof we need a little notation. We call any cuboid of $\bbZ^d$ with faces perpendicular to the lattice directions simply \emph{cuboid}. For a cuboid $R\subset\bbZ^d$ and $\eta\in\O_{\bbZ^d}$ we denote by $[\eta]_R$ the set of sites $x\in R$ which can become infected by legal updates (recall Section \ref{subsec:models}) only using infections in $R$. Equivalently, $[\eta]_R$ can be viewed as the set of sites eventually infected by $2$-BP with initial condition the set $\{x\in R:\h_x=0\}$. Note that $[\h]_R$ is a union of disjoint cuboids. For $x,y\in R$ we write $\{x\overset{R}\longleftrightarrow y\}$ for the event that $[\h]_R$ contains a cuboid containing $x$ and $y$.

The next proposition essentially states that the infection time is at least the inverse density of critical droplets.
\begin{prop}
\label{prop:lower:bound}
Let $V=[-\ell,\ell]^d$ with $\ell=\ell(q)$ be such that
\begin{equation}
\label{eq:lower:bound:easy}
    \mu_q(0\in[\h]_V)=o(1)
\end{equation}
and let
\begin{equation}
    \label{eq:lower:bound:condition}
\rho:= \sup_{x\in V: d(x,V^c)=1}
\mu_q\big(x\overset{V}\longleftrightarrow 0\big).
\end{equation}
Then
\[\bbE_{\mu_q}[\tau_0]\ge \frac{\O(1)}{\rho|V|^2}\]
and $\t_0\ge q/(\r|V|^2)$ w.h.p.
\end{prop}
\begin{proof}
Let $(\h(t))_{t\ge 0}$ denote the stationary KCM on $\bbZ^d,$ let $\cI=\{\o:0\in [\o]_V\}$ and let $\t=\inf\{t\ge 0,\eta(t)\in \cI\}.$ Given a configuration $\eta\in\O$, we say that the origin is infectable inside $V$ iff $\eta\in\cI$. The key observation here is that, by construction, $\tau_0\ge \t$. 

Suppose that at time $t=0$ the origin is not infectable inside $V$ or, equivalently, that $\tau>0$. Then we claim that at time $\tau >0$ there exists a site $x$ at the boundary of $V$ such that $\eta(\tau)\in\{x\overset{V}\longleftrightarrow0\}.$ In other words, at time $\tau$ a suitable very unlikely infection has appeared in $V$. To prove the claim assume $\t>0$ and consider the site $x\in\bbZ^d$ which is updated at time $\t$. Necessarily $x\in V$ and $d(x,V^c)=1$, since otherwise $[\eta^x(\t)]_V=[\eta(\t)]_V$, where $\eta^x(\t)$ is the configuration equal to $\eta(\t)$ except at site $x$. Furthermore, by definition of $\t$, $\eta(\t)\in\cI$ and $\eta^x(\t)\not\in\cI$. But this implies $\eta(\tau)\in\{x\overset{V}\longleftrightarrow0\}$, since otherwise a change of the state at $x$ could not change the infectability of the origin inside $V$.

Recall now the rate one Poisson clocks discussed at the beginning of Section~\ref{subsec:models} and let $N_V(s)$ denote the random number of clock rings (legal or not) at sites in $V$ up to time $s$. Let also $\eta^{(j)}$ denote the configuration right after the $j$-th clock ring. 
By the above we have
\[\bbP_{\m_q}(0<\tau\le s\tc N_V(s))\le\sum_{j=1}^{N_V(s)}\sum_{\substack{x\in V\\d(x,V^c)=1}}\bbP_{\m_q}\big(\eta^{(j)}\in\{x\overset{V}\longleftrightarrow0\}\tc N_V(s)\big).\]
Yet, conditionally on the clock rings, $\eta^{(j)}$ is distributed according to $\m_q$ for the stationary FA-2f process (see e.g.\ \cite{Hartarsky22univlower}*{Claim 3.11}). Hence, recalling \eqref{eq:lower:bound:condition}, we get
\begin{equation}
\label{eq:stationarity}
\bbP_{\m_q}(0<\t\le s\tc N_V(s))\le N_V(s)|V|\rho.
\end{equation}
Using $\bbE(N_V(s))=s|V|$, \eqref{eq:stationarity} gives
\begin{align*}
\bbP_{\m_q}(\t\le s)&{}= \bbP_{\m_q}(\t=0)+\bbE\big( \bbP_{\m_q}(0<\t\le s\tc N_V(s))\big)\\
&{}
\le o(1) + s|V|^2\rho,
\end{align*}
where $\bbE$ denotes the average w.r.t.\ $N_V(s)$ and we used  \eqref{eq:lower:bound:easy} to get $\bbP_{\m_q}(\t=0)=\m_q(\cI)=o(1)$. In conclusion, for all $\e>0$ we have
\[\limsup_{q\to0}\bbP_{\m_q}\big(\t_0\le\varepsilon/(|V|^2\r)\big)\le \limsup_{q\to0}\bbP_{\m_q}\big(\t\le\varepsilon/(|V|^2\r)\big)\le \e,\]
which concludes the proof by Markov's inequality.
\end{proof}

We can now easily deduce the lower bounds of Theorem \ref{th:FA2f} from Proposition \ref{prop:lower:bound} and the following bootstrap percolation results.
\begin{thm}[Eq.~(5.11) of \cite{Aizenman88}]
\label{th:AL}
For any $d\ge2$ there exists $c=c(d)>0$ such that \eqref{eq:lower:bound:easy} holds for any $d\ge 2$ and $\ell\le \exp(cq^{-1/(d-1)})$.
\end{thm}
\begin{thm}[Theorem~6.1, Lemma~3.9 and Eq.~(4) of \cite{Hartarsky19}]
\label{th:HM}
Let $d=2$ and $\ell= \frac{1}{4q'}\log(1/q')$, where $q'=-\log(1-q)$. Fix a cuboid (\emph{i.e.}\ rectangle) $R\subset \bbZ^2$ with side lengths $a,b$ such that $1\le a\le b\le 2\ell$ and $b\ge \ell$. Then 
\[\m_q([\eta]_R=R)\le\exp\Big(-\frac{\pi^2}{9q}+\frac{O(1)}{\sqrt{q}}\Big).\]
\end{thm}
\begin{thm}[Theorem~17 of \cite{Balogh12}]
\label{th:BBDCM}
Let $d\ge 2$ and $\e>0$. Let $C_0$ be sufficiently large depending on $d$ and $\e$. Then for any $q$ small enough, $C>C_0$ not depending on $q$ and cuboid $R$ with longest edge of length $\ell=C/q^{1/(d-1)}$ we have
\[\m_q([\h]_R=R)\le \exp\Big(-\frac{d\cdot\l(d,2)-\e}{q^{1/(d-1)}}\Big).\]
\end{thm}
\begin{proof}[Proof of the lower bounds \eqref{eq:FA2f:lower} and \eqref{eq:lbd} in Theorem \ref{th:FA2f}]
Fix $d$ and $\ell$ as in Theorem \ref{th:HM}. Theorem \ref{th:AL} implies \eqref{eq:lower:bound:easy}. Then Theorem~\ref{th:HM} and a union bound over all possible cuboids $R\subset V=[-\ell,\ell]^2$ containing both $0$ and some $x$ with $d(x,V^c)=1$ give $\rho\le\exp(-\frac{\pi^2}{9q}+\frac{O(1)}{\sqrt{q}})$. Thus, \eqref{eq:FA2f:lower} follows from Proposition \ref{prop:lower:bound}.

Fix $d$ and $\ell$ as in Theorem \ref{th:BBDCM}. Theorem \ref{th:AL} implies \eqref{eq:lower:bound:easy}. The upper bound on $\rho$ leading to \eqref{eq:lbd} follows from Theorem~\ref{th:BBDCM} together with a union bound as above, so we may conclude by Proposition \ref{prop:lower:bound}.
\end{proof}

\section{Constrained Poincar\'e inequalities}
\label{sec:Poinc}
In this section we state and prove various Poincar\'e inequalities for the auxiliary chains that will be instrumental in Section \ref{sec:droplet} (see Lemmas \ref{lem:bisectio} and \ref{lem:lin}). 

\subsection{Notation}
\label{subsec:notation}
Given $\L\subset \bbZ^2$ and $\omega\in \Omega$, we  write $\omega_{\Lambda}\in\Omega_{\Lambda}:=\{0,1\}^\L$ for the restriction of $\omega$ to $\Lambda$ and we denote by $\mu_\L$ the marginal of $\mu$ on $\O_\L$.  The configuration (in $\O$ or $\O_\L$) identically equal to one is denoted by $\mathbf 1$. Given disjoint $\Lambda_1,\L_2\subset \bbZ^2$, $\o^{(1)}\in \O_{\L_1}$ and $\o^{(2)}\in\O_{\L_2}$, we write $\omega^{(1)}\cdot\omega^{(2)}\in \Omega_{\Lambda_1\cup\Lambda_2}$ for the configuration equal to $\omega^{(1)}$ in $\Lambda_1$ and to $\omega^{(2)}$ in $\Lambda_2$. For $f:\Omega\to \mathbb R$ we will denote by $\mu(f)$ its expectation w.r.t.\ $\mu$ and by $\mu_{\L}(f)$ and $\var_\L(f)$ the mean and variance w.r.t.\ $\mu_\L$, given $\o_{\bbZ^2\setminus\L}$.

For sake of completeness, we recall the classic definitions of  Dirichlet form, Poincar\'e inequality, and relaxation time.
Given  a measure $\nu$ and a Markov process with generator $\mathcal L$ reversible w.r.t.\ $\nu$, the corresponding Dirichlet form 
$\cD:\Dom(\cL)\to\mathbb R$ is defined as  
 \begin{equation}\label{dirichlet}
\cD(f):=- \nu(f \cdot \cL f).
\end{equation}
For the FA-$2$f model, 
the definition of Section \ref{subsec:models} yields
the following Dirichlet form \begin{equation}\label{dirichlet:FA}
\cD^{\text{FA-$j$f}}_{\mathbb Z^d}(f)=\sum_{x\in \bbZ^d}\mu\left(c_{x} \mbox{Var}_x(f)\right).
\end{equation}
with $c_x$
the indicator function of the event \emph{``the constraint at $x$ is satisfied''}, namely
for $x\in\bbZ^2$ and $\eta\in\O$ we set 
\begin{equation}
    \label{def:ratesbis}
  c_x(\eta)=
  \begin{cases}
  1 & \text{ if $\sum_{y\sim x}(1-\eta_y)\geq 2$}
  \\  0 &  \text{ otherwise}
  \end{cases}
\end{equation}
where $y\sim x$ if $x,y$ are  nearest neighbours.

We say that a {\emph{Poincar\'e inequality}} with constant $C$ is satisfied by the Dirichlet form if
for any function $f\in \Dom(\cL)$ it holds
\begin{equation}
  \label{poincare}
  \var_{\nu}(f)\leq C \cD(f).
  \end{equation}
  Finally, the \emph{relaxation time} is defined as the best constant in the Poincar\'e inequality, namely
  \begin{equation}
  \label{eq:trel}
\trel:=\sup_{\substack{f\in \Dom(\cL)\\ \var_{\nu}(f)\neq 0}}\frac{\var_{\nu}(f)}{\cD(f)}.
\end{equation}
A finite relaxation time  implies that the reversible measure is mixing for the
semigroup $P_t:=e^{t\cL}$ with exponentially decaying correlations (see \emph{e.g.}\ \cite{Saloff-Coste97}), namely
for all $f\in L^2(\nu)$ it holds
 \begin{equation}\label{expo}\var_{\nu}\left(P_t f\right)={\nu}( f P_t f )-\nu(f)^2\leq \exp(-2 t/\trel)\var_{\nu}(f).\end{equation}

\subsection{FA-1f-type Poincar\'e inequalities}
\label{sec:KCMtools}
Fix  $\Lambda\subset\mathbb Z^2$  a connected set and let $\Omega^+_{\Lambda}=\Omega_\Lambda\setminus\mathbf 1$. Given $x\in \L$ let $N^\L_x$ be the set
of neighbours of $x$ in $\L$ and let $\cN_x^\L$ be the event that
$N_x^\L$ contains at least one infection. For any $z\in \L$ consider the two Dirichlet forms
\begin{align*}
  \cD_\L^{\text{FA-$1$f}}(f)&= \mu_{\Lambda}\Big(\sum_{x\in\Lambda}\1_{\cN_x^\L}\var_x(f)\tc\Omega^+_{\Lambda}\Big),&f:\O^+_\L&{}\to \bbR,\\
  \cD_\L^{\text{FA-$1$f},z}(f)&= \mu_{\Lambda}\Big(\sum_{\substack{x\in\Lambda\\x\neq z}}\1_{\cN_x^\L}\var_x(f)+\var_z(f)\Big),& f:\O_\L&{}\to \bbR.
\end{align*}
\begin{rem}
The alert reader will recognise the above expressions as the Dirichlet forms of the FA-$1$f process on $\O^+_\L$ or on $\O_\L$ with the site $z$ unconstrained.
\end{rem}
Our first tool is a Poincar\'e inequality for these Dirichlet forms.
\begin{prop}
\label{prop:FA1f}\leavevmode 
Let $\Lambda$ be a connected subset of $\mathbb Z^2$ and let $z\in \L$ be an arbitrary site. Then:
\begin{enumerate}
\item\label{item:FA1ferg}
for any $f:\O^+_\L\to \bbR$,
\begin{equation}
  \label{eq:FA1fergodic}
\var_{\Lambda}(f \tc \Omega^+_{\Lambda})
\leq\frac{1}{q^{O(1)}}\cD_\L^{\mathrm{FA-1f}}(f);
\end{equation}
\item\label{item:FA1fboundary} for any $f:\O_\L\to \bbR$,
\begin{equation}
    \label{eq:FA1fboundary}
    \var_{\Lambda}(f)
  \leq\frac{1}{q^{O(1)}}\cD_\L^{\mathrm{FA-1f},z}(f),
\end{equation}
\end{enumerate}
where the constants in the $O(1)$ do not depend on $z$ or $\L$.
\end{prop}
\begin{proof}
Inequality \eqref{eq:FA1fergodic} is proved in
\cite{Blondel13}*{Theorem 6.1}. In order to prove \eqref{eq:FA1fboundary}, consider the auxiliary Dirichlet form 
\[
\m_\L(\var_z(f))+ \mu_{ \L}\big(\1_{\O^+_\L}\var_\L(f\tc \O^+_\L)\big).
\]
The corresponding ergodic, continuous time Markov chain on $\O_\L$, reversible w.r.t.\ $\mu_\L$, updates the state of $z$ at rate 1 and, if $\o\in\O^+_\L$, it updates the entire configuration w.r.t.\ $\pi(\cdot\tc \O^+_\L)$. Observe that two copies of this chain attempting the same updates simultaneously couple as soon as they update the state of $z$ to state 0 and then change to the same configuration in $\O_\L^+$. Thus, by \cite{Levin09}*{Corollary 5.3 and Theorem 12.4} the relaxation time of this chain is $O(1/q)$, as the first step occurs at rate $q$. Indeed, after time $1/q$ there is probability $\O(1)$ that the above sequence of two consecutive updates has been performed.

Hence,
\begin{align*}
\var_\L(f)&\le O(1/q) \Big(\m_\L(\var_z(f))+\mu_{ \L}\big(\1_{\O^+_\L}\var_\L(f\tc \O^+_\L)\big)\Big)\\
&\le \frac{1}{q^{O(1)}}\Big(\m_\L(\var_z(f))+ \mu_\L(\O^+_\L)\cD_\L^{\mathrm{FA-1f}}(f)\Big),
\end{align*}
where the second inequality follows from \eqref{eq:FA1fergodic}.  
We may then conclude by observing that $\m_\L(\var_z(f))+\mu_\L(\O^+_\L)\cD_\L^{\mathrm{FA-1f}}(f)\le 2\cD_\L^{\mathrm{FA-1f},z}(f)$.
\end{proof}
Our second tool is a general constrained Poincar\'e inequality
for two independent random variables.
\begin{prop}[See \cite{Hartarsky21a}*{Lemma 3.10}]
\label{prop:2block}
Let $X_1,X_2$ be two independent random variable taking values in two
finite sets $\bbX_1,\bbX_2$ respectively.  Let also $\cH\subset \bbX_{1}$ with
$\bbP(X_1\in \cH)>0$. Then for any $f:\bbX_1\times \bbX_2\to \bbR$ it holds
\[
\var(f)\le  2\bbP(X_1\in \cH)^{-1}\bbE\big(\var_{1}(f)
+\1_{\{X_1\in \cH \}}
\var_{2}(f)\big). 
\]
with $\var_i(f)=
\var(f(X_1,X_2)\tc X_i)$.
\end{prop}
Roughly speaking, this states that the chain that updates $X_1$ at rate 1 and $X_2$ at rate 1 only if $\cH$ occurs, has relaxation time given by the inverse probability of $\cH$.

\subsection{Constrained block chains}
\label{sec:auxiliary}
In this section we define two auxiliary constrained reversible Markov
chains and give an upper bound for the corresponding Poincar\'e constants (Propositions \ref{key:bisection} and \ref{key:bisection2}).

Let $(\O_i,\pi_i)_{i=1}^3$ be finite probability spaces and let
$(\O,\pi)$ denote the associated product space. For $\omega\in\O$ we
write $\o_i\in \O_i$ for its $i\textsuperscript{th}$ coordinate and we assume for simplicity that
$\pi_i(\o_i)>0$ for each $\o_i$. Fix $\mathcal A_3\subset\O_3$ and 
for each $\omega_3\in \mathcal A_3$ consider an event $\mathcal
B_{1,2}^{\omega_3}\subset\O_1\times \O_2$. Analogously, fix $\mathcal A_1\subset\O_1$ and 
for each $\omega_1\in \mathcal A_1$ consider an event $\mathcal
B_{2,3}^{\omega_1}\subset\O_2\times \O_3$. We then set 
  \begin{align*}
 \cH &{}=\big\{\o:{\omega_3\in \mathcal A_3}\text{ and }
             (\omega_1,\omega_2)\in \mathcal
             B_{1,2}^{\omega_3}\big\},\\
    \cK&{}=\big\{\o:{\omega_1\in \mathcal A_1} \text{ and } (\omega_2,\omega_3)\in \mathcal B_{2,3}^{\omega_1}\big\}
  \end{align*}
and let for any $f:\cH\cup \cK\to \bbR$ 
\[\Da^{(1)}(f) =\pi\big(\1_{\cH}\var_\pi(f\tc
   \cH,\o_3) + \1_{\cK}\var_\pi(f\tc \cK,\o_1)\tc \cH\cup \cK \big).\]
\begin{obs}
\label{obs:aux1}
It is easy to check that $\Da^{(1)}(f)$ is the Dirichlet form of the
continuous time Markov chain on $\cH\cup \cK$ in which if $\o\in \cH$ the
pair $(\o_1,\o_2)$ is resampled with rate one from
$\pi_1\otimes\p_2(\cdot \tc \cB_{1,2}^{\o_3})$ and if $\o\in \cK$ the
pair $(\o_2,\o_3)$ is resampled with rate one from
$\pi_2\otimes\p_3(\cdot \tc \cB_{2,3}^{\o_1})$. This chain is
reversible w.r.t.\ $\pi(\cdot\tc \cH\cup \cK)$ and its constraints, contrary to
what happens for general KCM, depend on the to-be-updated variables. 
\end{obs}
\begin{prop}\label{key:bisection}
There exists a universal constant $c$ such that the following holds. Suppose that there exist two events
$\cF_{1,2}$, $\mathcal F_{2,3}$
such that 
\begin{align}
\big\{\omega:\omega_3\in \cA_3\text{ and } (\omega_1,\omega_2)\in  \cF_{1,2}\big\}&{}\subset \cH\cap\cK,\label{H1}\\
\big\{ \omega: \omega_1\in \cA_1  \text{ and }
  (\omega_2,\omega_3)\in\cF_{2,3} \big\}&{}\subset \cH\cap\cK\label{H2}
\end{align}
and let
\[
  \Ta^{(1)}=  \max_{\omega_3\in \cA_3, }\Big(\frac{\pi(
    \cB_{1,2}^{\o_3})}{\pi(\cF_{1,2})}\Big)^{2}\max_{\o_1\in \cA_1}\frac{\pi( \cB_{2,3}^{\o_1})}{\pi(\cF_{2,3})}.
\]
Then, for all $f:\cH\cup \cK\to \bbR$,
\[
  \var_\pi(f\tc \cH\cup \cK)\le c
  \Ta^{(1)}\Da^{(1)}(f).
\]
\end{prop}
\begin{proof}
Consider the Markov chain $(\o(t))_{t\ge 0}$ determined by the Dirichlet form $\Da^{(1)}$ as described in Observation \ref{obs:aux1}.
Given two arbitrary initial conditions $\o(0)$ an $\o'(0)$ we will
construct a coupling of the two chains such that with probability $\O(1)$ we have $\o(t)=\o'(t)$
for any $t>\Ta^{(1)}$. Standard arguments (see for example \cite{Levin09}*{Theorem 12.4 and Corollary 5.3})
then prove that for this chain it holds $\trel=O(\Ta^{(1)})$ and the conclusion of the proposition follows. 
To construct our coupling,  we use the following representation of the Markov chain. 
We are given two independent Poisson clocks with
rate one and the chain transitions occur only at the clock rings. Suppose that the first clock rings. If the current configuration
$\o$ does not belong to $\cH$ the ring is ignored.  Otherwise, a Bernoulli variable
$\xi$ with probability of success $\pi(\cF_{1,2}\tc \cB_{1,2}^{\o_3})$
is sampled. If
$\xi=1$, then the pair $(\o_1,\o_2)$ is resampled w.r.t.\ the measure $\pi(\cdot\tc
\cF_{1,2},\cB_{1,2}^{\o_3})$, while if $\xi=0$, then $(\o_1,\o_2)$ is
resampled w.r.t.\ the measure $\pi(\cdot\tc
\cF^c_{1,2},\cB_{1,2}^{\o_3})$. Clearly, in doing so 
the couple $(\o_1,\o_2)$ is resampled w.r.t.\ $\pi(\cdot\tc
\cB_{1,2}^{\o_3})$. Similarly if the second clock rings but with $\cH$,
$(\o_1,\o_2)$, $\cF_{1,2}$ and $\cB_{1,2}^{\o_3}$ replaced by $\cK$,
$(\o_2,\o_3)$, $\cF_{2,3}$ and $\cB_{2,3}^{\o_1}$ respectively.
It is important to notice that  $\pi(\cdot\tc
\cF_{1,2},\cB_{1,2}^{\o_3})=\pi(\cdot\tc
\cF_{1,2})$ for all
$\o_3\in \cA_3$, as, by assumption, $\cF_{1,2}\subset
\bigcap_{\o_3\in \cA_3}\cB_{1,2}^{\o_3}$. Similarly, $\pi(\cdot\tc
\cF_{2,3},\cB_{2,3}^{\o_1})=\pi(\cdot\tc
\cF_{2,3})$ for all $\o_1\in \cA_1$.

In our coupling both chains use the same clocks. Suppose that the
first clock rings and that the current pair of configurations is
$(\o,\o')$. Assume also  that at
least one of them, say $\o$, is in $\cH$ (otherwise, both remain unchanged). In order
to construct the coupling update we proceed as follows.
\begin{itemize}
\item If $\o'\notin
\cH$ then $\o$ is updated as described above, while $\o'$ stays
still.
\item If $\o'\in \cH$ we first maximally couple the two Bernoulli variables $\xi,\xi'$ corresponding to $\o,\o'$ respectively. Then:
\begin{itemize}
\item if $\xi=\xi'=1$, we update both $(\o_1,\o_2)$ and $(\o'_1,\o'_2)$
  to the \emph{same} couple $(\eta_1,\eta_2)\in \cF_{1,2}$ with
  probability $\pi((\eta_1,\eta_2)\tc \cF_{1,2});$ 
\item  otherwise we resample $(\o_1,\o_2)$ and $(\o'_1,\o'_2)$ independently from their respective law given $\xi,\xi'$.
\end{itemize}
\end{itemize}
Similarly if the ring comes from the second clock. The final coupling is then equal to the Markov chain on $\O\times \O$ with the transition rates described above. 
Suppose now that there are three consecutive rings occurring at times $t_1<t_2<t_3$ such that:
\begin{itemize}
\item
the first and last ring come from the first clock while the second ring comes from the second clock, and
\item the sampling of the Bernoulli variables (if any) at times $t_1$, $t_2$ and $t_3$ all produce the value one. 
\end{itemize}
Then we claim that at time $t_3$ the two copies are coupled.

To prove the claim, we begin by observing that after the first update at $t_1$ both copies
of the coupled chain belong to $\cK$. Here we use \eqref{H1}. Indeed, if the first update is successful for $\o$ (\emph{i.e.}\ $\o\in \cH$) then the updated configuration belongs to $\cF_{1,2}\times\{\o_3\}\subset \cK$, because of our assumption $\xi=1$. If, on the contrary, the first update fails (\emph{i.e.}\ $\o\not\in \cH$) then $\o\in \cK\setminus \cH$ before and after the update. The same applies to $\o'$.

Next, using again the assumption on the Bernoulli variables together
with the previous observation, we get that after the second ring the
new pair of current configurations agree on the second and third
coordinate. Moreover both copies belong to $\cH$ thanks to \eqref{H2}. Finally, after the third ring the two copies couple on the first
and second coordinates using
again the assumption on the outcome for the Bernoulli variables.

In order to conclude the proof of the proposition it is enough to observe that for any given time interval $\D$ of length one the probability that  there exist $t_1<t_2<t_3$ in $\D$ satisfying the requirements of the
claim is bounded from below by
\[
c\min_{\omega_3\in\cA_3}
\pi\big(\mathcal F_{1,2}\tc \cB_{1,2}^{\omega_3}\big)^{2} \min_{\omega_1\in\cA_1}
\pi\big(\mathcal F_{2,3}\tc \cB^{\omega_1}_{2,3}\big),
\]
for some constant $c>0$. 
\end{proof}
In the same setting consider 
 two other events $\cC_{1,2}\subset\O_1\otimes\O_2$, $\cC_{2,3}\subset\O_2\otimes \O_3$ and let
\begin{align*}
  \cM&{}=\cA_3\cap \cC_{1,2},&\cN  &{}=\cA_1\cap
  \cC_{2,3}.
\end{align*}
The Dirichlet form of our second Markov chain on $\cM\cup \cN$  is then
\begin{multline}
  \label{Dir2}
 \Da^{(2)}(f) =\pi\Big(\1_{ \cM}\var(f\tc \cC_{1,2},\o_3)+\1_{ \cM}\var(f\tc \cA_3,\o_1,\o_2)\\
   + \1_{ \cN}\var(f\tc \cC_{2,3},\o_1)+\1_{ \cN}\var(f\tc \cA_1,\o_2,\o_3)\tc \cM\cup \cN\Big).
\end{multline}
\begin{obs}
Similarly to the first case, the continuous time chain defined by \eqref{Dir2} is reversible
w.r.t.\ $\pi(\cdot\tc \cM\cup \cN)$ and it can be described as
follows. If $\o\in \cM$ then with rate one $(\o_1,\o_2)$ is
resampled w.r.t.\ $\pi_1\otimes\pi_2(\cdot\tc \cC_{1,2})$ and, independently at unit rate, $\o_3$ is
resampled w.r.t.\ $\pi_3(\cdot\tc \cA_3)$. Similarly, independently from the previous updates at rate one, if $\o\in \cN$ then
$(\o_2,\o_3)$ is resampled w.r.t.\ $\pi_2\otimes\pi_3(\cdot\tc
\cC_{2,3})$ and, independently, $\o_1$ is resampled from $\pi_1(\cdot\tc \cA_1)$.       
\end{obs}
\begin{prop}\label{key:bisection2}
There exists a universal constant $c$ such that the following holds. Suppose that  there exist an event $\hat \cC_{1,2}\subset \cC_{1,2}$
and a collection
$(\cA_3^{\omega_1,\omega_2})_{(\omega_1,\omega_2)\in
  \hat\cC_{1,2}}$ of subsets of $\cA_3$ such that 
\begin{equation}
\label{Hbar} 
\big\{\omega :(\omega_1,\omega_2) \in \hat\cC_{1,2}\text{ and }
  \o_3\in \cA_3^{\o_1,\o_2}\big\}\subset \cM\cap \cN,
\end{equation}
and let
\[\Ta^{(2)}= \max_{(\o_1,\o_2)\in \hat\cC_{1,2}}\frac{\pi(\cA_3)}{\pi(\cA_3^{\o_1,\o_2})} \times \frac{\pi( \cC_{1,2})}{\pi( \hat\cC_{1,2})}.
\]
Then there exists $c>0$ such that for all $f:\cM\cup \cN\to \bbR$, 
\[
  \var(f\tc \cM\cup \cN)\le c\Ta^{(2)} \Da^{(2)}(f).
  \]
\end{prop}
\begin{proof}
We proceed as in the proof of Proposition \ref{key:bisection} with
  the following representation for the Markov chain.
  We are given four independent Poisson clocks of rate one and each clock comes equipped with a collection of i.i.d.\ random variables. The four independent collections, the first being for the first clock etc, are
  \[
    \big((\o^{(i)}_1, \o^{(i)}_2)\big)_{i=1}^\infty, \quad \big(\eta^{(i)}_3\big)_{i=1}^\infty, \quad\big((\o^{(i)}_2, \o^{(i)}_3)\big)_{i=1}^\infty,\quad \big(\eta^{(i)}_1\big)_{i=1}^\infty,
  \]
where the laws of the collections are $\pi_1\otimes\pi_2(\cdot\tc \cC_{1,2})$, $\pi_3(\cdot\tc \cA_3)$, $\pi_2\otimes\pi_3(\cdot\tc \cC_{2,3})$ and $\pi_1(\cdot\tc  \cA_1)$ respectively. 

At each ring of the first and second clocks the configuration is updated with the variables from the corresponding collection iff $\o\in\cM$. Similarly for the third and fourth clocks with $\cN$. In order to couple different initial conditions, we use the same collections of clock rings and update configurations.
  
Suppose now that there are four consecutive rings $t_1<t_2<t_3<t_4$, coming from the first, second, third and fourth clocks in that order, such that:
\begin{itemize}
\item at $t_1$ the proposed update $(\eta_1,\eta_2)$ of the first two coordinates belongs to $\hat \cC_{1,2}$, and
\item at $t_2$ the proposed update $\eta_3$ of the third coordinate belongs to $ \cA_3^{(\eta_1,\eta_2)}$.
\end{itemize}
We then claim that after $t_4$ all initial conditions $\o$ are coupled. To prove
this, we first observe that after the second ring each chain
belongs to $\cN$. Indeed, if $\o\not\in\cM$, then
the first two proposed updates are ignored and the configuration $\o\in\cN\setminus\cM$. If, on the contrary, $\o\in\cM$, then both updates are successful and the configuration is updated to $(\eta_1,\eta_2,\eta_3)\in\hat\cC_{1,2}\times\cA^{\eta_1,\eta_2}_3\subset \cM\cap\cN$ by \eqref{Hbar}. 

Since after $t_2$ the state of the chain is necessarily in $\cN$, the third and fourth updates to states $(\eta'_2,\eta'_3)$ and $\eta'_1$ respectively are both successful and thus any initial condition leads to the state $(\eta'_1,\eta'_2,\eta'_3)$ after $t_4$, which proves the claim. The proof is then completed as in Proposition \ref{key:bisection}.
\end{proof}

\section{Mobile droplets}
\label{sec:droplet}
This section, which represents the core of the paper,
is split into two parts: 
\begin{itemize}
\item the definition of mobile droplets together with the choice of the mesoscopic critical length scale $\Ld$ characterising their linear size;
\item  the analysis of two key properties of mobile droplets namely: \begin{itemize}
  \item their  equilibrium probability $\rd$;
  \item the relaxation time of FA-$2$f in a box of linear
    size $\Theta(L_D)$ \emph{conditionally} on the presence of a
    mobile droplet.
\end{itemize}
\end{itemize}
Mobile droplets are defined as boxes of suitable
linear size in which the configuration of infection is \emph{super-good} (see Definition \ref{def:droplets}). In turn, the super-good event (see Section
\ref{sec:construction}) is constructed recursively via a multi-scale
procedure on a sequence of exponentially increasing length scales
$(\ell_n)_{n=1}^N$ (see Definition \ref{def:scales}). 
While clearly inspired by the classical procedure used in bootstrap percolation \cite{Holroyd03}, an important 
novelty in our construction is the freedom that we allow for the position of the super-good core of scale $\ell_n$ inside the super-good region of
scale $\ell_{n+1}$.
The final scale
$\ell_N$ corresponds to the critical scale $\Ld$ mentioned above and a
convenient choice is $\Ld\sim q^{-17/2}$ (see
\eqref{defellN}). There is nothing special in the exponent $17/2$:  as
long as we choose a sufficiently large exponent our results would not
change. The choice of $L_D$ is in fact only dictated by the
requirement that w.h.p.\ there exist no $\Ld$
consecutive lattice sites at distance $\exp(\log^{O(1)}(1/q)/q)$ from the
origin which are healthy and $\Ld= e^{o(1/q)}$.
Finally, similarly to their bootstrap percolation counterparts, the probability $\rd$ of mobile droplets crucially satisfies  $\rd\simeq (\tbp)^{-2}$ (see Proposition \ref{prop:probadroplet}) and in general for FA-$2$f in dimension $d$ it satisfies $\rd\simeq (\tbp)^{-d}$.

The
extra degree of freedom in the construction of the super-good event provides a much more flexible structure that
can be moved around using the FA-$2$f moves \emph{without} going through the
bottleneck corresponding to the creation of a brand new additional droplet
nearby. The main consequence of this feature (see Proposition \ref{thm:droplet}) is that the relaxation
time of the FA-$2$f dynamics in a box of side $\Ld$ conditioned
on being super-good is sub-leading w.r.t.\ $\rd^{-1}$ as $q\to 0$
and it contributes only to the second order term in Theorem \ref{th:FA2f}.

\subsection{Notation}
\label{sec:notation}

For any integer $n$, we write $[n]$ for the set $\{1,\dots,n\}$.  We denote by $\vec e_1, \vec e_2$ the standard basis of $\bbZ^2$, and write $d(x,y)$ for the Euclidean distance between $x,y\in \bbZ^{{2}}$. Given a set $\Lambda\subset \bbZ^2$, we set $\partial \Lambda:=\{y\in \bbZ^{d}\setminus \Lambda,d(y,\L)=1\}$.
Given two positive integers $a,b$, we write
$R(a,b)\subset\mathbb Z^2$ for the rectangle $[a]\times[b]$ and we refer to $a,b$ as the \emph{width} and \emph{height} of $R$ respectively. We also write $\partial_r R$ ($\partial_l R$) for the column $\{a+1\}\times [b]$ (the column $\{0\}\times [b]$), and $\partial_u R$ ($\partial_d R$) for the the row $[a]\times\{b+1\}$ (the row $[a]\times\{0\}$). Similarly for any rectangle of the form $R+x,x\in \bbZ^2$.

Given $\L\subset \bbZ^2$ and $\omega\in \Omega$, we  write $\omega_{\Lambda}\in\Omega_{\Lambda}:=\{0,1\}^\L$ for the restriction of $\omega$ to $\Lambda$. The configuration (in $\O$ or $\O_\L$) identically equal to one is denoted by $\mathbf 1$. Given disjoint $\Lambda_1,\L_2\subset \bbZ^2$, $\o^{(1)}\in \O_{\L_1}$ and $\o^{(2)}\in\O_{\L_2}$, we write $\omega^{(1)}\cdot\omega^{(2)}\in \Omega_{\Lambda_1\cup\Lambda_2}$ for the configuration equal to $\omega^{(1)}$ in $\Lambda_1$ and to $\omega^{(2)}$ in $\Lambda_2$. We write $\mu_\L$ for the marginal of $\mu_q$ on $\O_\L$ and $\var_\L(f)$ for the variance of $f$ w.r.t.\ $\mu_\L$, given the variables $(\o_x)_{x\notin \L}$. 

\subsection{Super-good event and mobile droplets}
\label{sec:construction}
As anticipated, mobile droplets will be square regions of a certain side length in which the infection configuration satisfies a
specific condition dubbed \emph{super-good}. The latter requires in turn the definition of a key event for rectangles---\emph{$\o$-traversability} (see also
\cite{Holroyd03})---together with a sequence of exponentially increasing length scales. 
\begin{figure}
\centering
  \begin{tikzpicture}[>=latex,x=0.5cm,y=0.5cm, scale=1.5]
\node at (0.7,0.5) {$R$};
\foreach \i in {2,...,12}
{\draw[gray,opacity=0.5]  (-3,0.5*\i)-- (1,0.5*\i);}

\foreach \j in {2,...,10}
{\draw[gray]  (-4+0.5*\j,1)-- (-4+0.5*\j,6);}
\draw [->,very thick, dashed] (-3,3.8) --
(1,3.8);
\draw [->,very thick] (-0.8,6) -- (-0.8,1);
\fill (-2.5,0.495 * 3-0.5) circle (0.5mm);
\fill (-2,0.495 * 3+4) circle (0.5mm);
\fill (-1.5,0.495 * 3+1) circle (0.5mm);
\fill (-1.5,0.495 * 3+3) circle (0.5mm);
\fill (-0.5,0.495 * 3) circle (0.5mm);
\fill (0.5,0.495 * 5) circle (0.5mm);
\foreach \i in {2,...,12}
{
\fill (1.5,0.495 * \i+0.01) circle (0.5mm);
}
\foreach \i in {1,...,3}
{
\fill (-1.5,0.495 * \i +2) circle (0.5mm);
}
\end{tikzpicture}
\caption{
Black circles denote infected sites. The boundary condition $\o$ in the figure is fully infected on $\partial_r R$ and fully healthy elsewhere. The rectangle $R$ is $\o$-right-traversable (i.e $\cT_\rightarrow^\o(R)$ occurs) but it is neither $\o$-up-traversable, nor $\o$-left-traversable. It is also down-traversable ($\cT_\downarrow(R)$ occurs) but not traversable in any other direction.
}
\label{fig:traversability}
\end{figure}
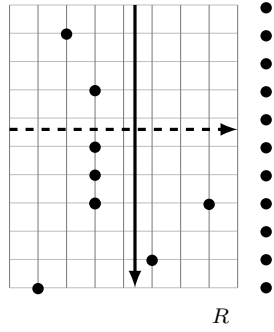

\begin{defn}[$\o$-Traversability]
\label{def:doubleg}
Fix a rectangle $R=R(a_1,a_2)+x$ together with $\eta\in \O_R$ and a boundary configuration $\omega\in\Omega_{\partial R} $. 
We say that $R$ is \emph{$\o$-right-traversable} for $\eta$ if each pair of adjacent columns of $R\cup \partial_r R$ contains at least one infection in $\eta\cdot\o$ (see Figure \ref{fig:traversability}). We denote this event by ${\mathcal T}^\o_{\rightarrow}(R)\subset\O_R$.

We say that $R$ is \emph{right-traversable} for $\eta$ if it is $\mathbf 1$-right-traversable or, equivalently, if it is $\o$-right-traversable for all $\omega$. We denote this event by $\mathcal T_{\rightarrow}(R)\equiv \cT_\rightarrow^{\mathbf 1}(R)\subset\O_R$. 

Up/left/down-traversability and $\o$-up/left/down-traversability is defined identically up to rotating $\eta$ and $\o$ appropriately (see Figure \ref{fig:traversability}).
\end{defn}
In figures we depict traversability by solid arrows and $\o$-traversability by dashed arrows (see Figure \ref{fig:traversability}). Notice that right-traversability requires that the rightmost column contains an infection. Similarly for the other directions.

\begin{defn}[Length scales and nested rectangles] 
\label{def:scales}
For all integer $n$ we set\footnote{This choice of 
geometrically increasing length scales is inspired by \cite{Gravner12}.}
\begin{equation}\label{def:fn}
  \ell_m=
  \begin{cases}
1 & \text{ if $m=0$},\\
\big\lfloor \frac{\exp{(m\sqrt q)}}{\sqrt q}\big\rfloor & \text{ if $m \geq 1$}    
  \end{cases}
\end{equation}
and 
\begin{equation}\label{def:Lambdan}
\Lambda^{(n)}=
\begin{cases}
R(\ell_{n/2},\ell_{n/2}) & \text{if $n$ is even,}\\
R(\ell_{(n+1)/2},\ell_{(n-1)/2})& \text{if $n$ is odd,}
\end{cases}
\end{equation}
(see Figure \ref{fig:super-good}). We say that a rectangle $R$ is of \emph{class $n$} if there exist $w,z\in\bbZ^2$ such that $\Lambda^{(n-1)}+w\subsetneq R\subset \Lambda^{(n)}+z$. We refer to single sites as rectangles \emph{of class $0$}.
\end{defn}
Note that $(\Lambda^{(2m)})_{m\geq 0}$ is a sequence of squares, while $(\Lambda^{(2m+1)})_{m\geq 0}$ is a sequence of rectangles elongated horizontally and $\Lambda^{(n_1)}\subset \Lambda^{(n_2)}$ if $n_1<n_2$. Moreover, for $n=2m> 0$, a rectangle of class $n$ is a rectangle of width $\ell_m$ and height $a_2\in (\ell_{m-1},\ell_m]$ and for  $n=2m+1$ it is a rectangle of height $\ell_m$ and width $a_1\in (\ell_m,\ell_{m+1}]$.

We are now ready to introduce the key notion of the $\o$-super-good event on different scales. This event is defined recursively on $n$ and it has a hierarchical structure. Roughly speaking, a rectangle $R$ of the form $R=\L^{(n)}+x, x\in \bbZ^2$, is $\o$-super-good if it contains a $\mathbf 1$-super-good rectangle $R'$ of the form $R'=\L^{(n-1)}+x'$ called the \emph{core} and outside the core it satisfies certain $\o$-traversability conditions (see Figure \ref{fig:super-good}). 

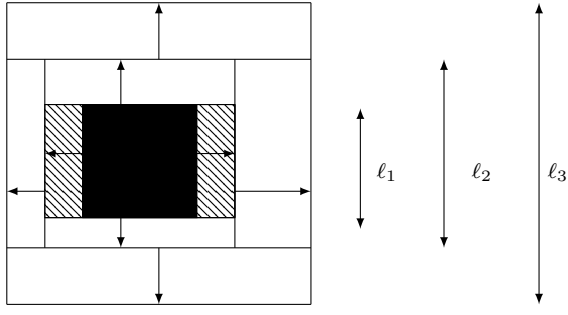
\begin{figure}
  \centering
  \begin{tikzpicture}[>=latex,x=0.5cm,y=0.5cm]
\draw[fill] (10,3.3) rectangle (13,6.3);
\draw[pattern=north west lines] (9,3.3) rectangle (14,6.3); 
\draw [<->] (22,1)--(22,9);
\node at (22.5,4.5) {$\ell_3$};
\draw [<->] (17.3,3)--(17.3,6.2);
\node at (18,4.5) {$\ell_{1}$};
\draw [<->] (19.5,2.5)--(19.5,7.5);
\node at (20.5,4.5) {$\ell_2$};
\draw  (8, 1)-- (8,9);
\draw  (16,1)-- (16,9);
\draw  (8, 1)-- (16,1);
\draw  (8, 9)-- (16,9);
\draw  (8, 7.5)-- (16,7.5);
\draw  (8, 2.5)-- (16,2.5);
\draw [->] (12,7.5) -- (12,9);
\draw [->] (11,6.3) -- (11,7.5);
\draw [->] (11,3.3) -- (11,2.5);
\draw [->] (12,2.5) -- (12,1);
\draw [->] (14,4) -- (16,4);
\draw [->] (9,4) -- (8,4);
\draw [->] (10,5) -- (9,5);
\draw [->] (13,5) -- (14,5);
\draw  (9, 2.5)-- (9,7.5);
\draw  (14, 2.5)-- (14,7.5);
\draw  (9, 3.3)-- (14,3.3);
\draw  (9, 6.3)-- (14,6.3);
\draw  (10, 3.3)-- (10,6.3);
\draw  (14, 3.3)-- (14,6.3);
\end{tikzpicture}
\caption{An example of super-good configuration in the square $\Lambda^{(6)}$. The black square, of the form $\Lambda^{(2)}+x$, is completely infected and it is a super-good core for the rectangle of the form $\L^{(3)}+x$ formed by it together with the two hatched rectangles. This rectangle of the form $\L^{(3)}+x$ is also super-good because of the right/left-traversability of the hatched parts (arrows) and it is a super-good core for the square containing it and so on.}
\label{fig:super-good}
\end{figure}

\begin{defn}[$\o$-Super-good rectangles]
\label{def:supergood}
Let us fix an integer $n\geq 0$, a rectangle $R=R(a_1,a_2)+x$ of class $n$ and $\omega\in\O_{\partial R}$. 
We say that $R$ is \emph{$\o$-super-good} for $\eta\in\O_{R}$ and denote the corresponding event by $\SG^\o(R)$ if the following occurs in $\eta\cdot\o$.
\begin{itemize}
\item $n=0$. In this case $R$ consists of a single site and
  $\SG^\o(R)$ is the event that this site is infected.
\item $n=2m$. For any $s\in [0,\ell_{m}-\ell_{m-1}]$ write $R=C_s\cup (\Lambda^{(n-1)}+ x+s\vec e_2)\cup D_s$,  where $C_s$ ($D_s$) is the part of $R$ below (above) $\L^{(n-1)}+x +s\vec e_2$. With this notation we set
\[\SG^\o_s(R):=\mathcal{T}^{\o}_{\downarrow}(C_s)\cap\SG^{\bf 1}(\Lambda^{(n-1)}+x+s\vec e_2)  \cap      \mathcal{T}^{\o}_{\uparrow}(D_s)\]
and let $\SG^\o(R)=\bigcup_{ s\in    [0,\ell_{m}-\ell_{m-1}]}\SG^\o_s(R)$.
\item $n=2m+1$. In this case $\SG^\o(R)$ requires that there is a core in $R$ of the form $\Lambda^{(n-1)}+ x+s\vec e_1, s\in [0,\ell_{m+1}-\ell_{m}]$, which is
$\mathbf 1$-super-good, and the two remaining rectangles forming $R$ to the left and to the right of the core are $\o$-left-traversable and $\o$-right-traversable respectively.
   \end{itemize}
We will say that $R$ is \emph{super-good} if it is  $\mathbf 1$-super-good and denote the corresponding event by $\SG(R)$.
\end{defn}
Note that $\SG^\o(R)$ is monotone in the boundary condition in the sense that if $R$ is super-good then $R$ is $\o$-super-good for all $\o\in\O_{\partial R}$. In order to make notation more concise, whenever a $\SG$ event appears in an average or a variance with respect to a rectangle $R$, we leave out the argument $R$ of the $\SG$ event, unless confusion arises. For example, $\m_R(\SG)$ will stand for $\m_R(\SG(R))$.
\begin{rem}[Irreducibility of the FA-$2$f chain in $\SG^\o(R)$]
\label{rem:irreducibility}
It is not difficult to verify that for all $\eta\in \SG^\o(R)$, there exists a sequence of legal updates that transforms $\eta$ into the fully infected configuration. Since the FA-$2$f dynamics is reversible, the above property implies that the FA-$2$f chain in $R$ restricted to $\SG^\o(R)$ is irreducible.
\end{rem}

Now let
\begin{equation}
\label{defn}    
N:=\Big\lceil \frac{8\log(1/q)}{\sqrt q}\Big\rceil
\end{equation}
and
observe that
\begin{equation}
\label{defellN}
 \ell_{N}={q^{-17/2+o(1)}}.\end{equation}

\begin{defn}[Mobile droplets]\label{def:droplets}
Given $\o\in \O$, a \emph{mobile droplet} for $\o$ is any square $R$
of the form $R=\Lambda^{(2N)}+x$ for which
$\o_R\in \SG(R)$. We set $\rd=\mu_{\Lambda^{(2N)}}(\SG)$ to be the probability of a mobile droplet.
\end{defn}
The first key property of mobile droplets we will need is the following.
\begin{prop}[Probability of mobile droplets]\label{prop:probadroplet}
For all $n\le 2N$,
\[\mu_{\Lambda^{(n)}}
(\SG)
\geq \exp\Big(-\frac{\pi^2}{9q}\big(1+O\big(\sqrt{q}\log^{2}(1/q)\big)\big)\Big).
\]
In particular, this lower bound holds for $\rd$.
\end{prop}
The proof of Proposition \ref{prop:probadroplet} follows from standard
2-BP techniques and it is deferred to Appendix \ref{app:BP}. The
second property of mobile droplets requires a bit of preparation. 

For $\Lambda\subset\mathbb Z^2$, $\omega\in\Omega_{\bbZ^2\setminus\Lambda}$, $\eta\in\O$ and $x\in \L$ we denote by
\[c_x^{\Lambda,\omega}(\eta)=c_x(\eta_\Lambda\cdot \o)\]
with $c_x$ defined in \eqref{def:ratesbis}, 
so that $c_x^{\Lambda,\omega}$ encodes the constraint at $x$ in $\Lambda$ with boundary condition $\o$. Given a rectangle $R$ of class $n$ and $\omega\in\O_{\mathbb Z^2\setminus R}$, let $\gamma^\o(R)$ be the smallest constant\footnote{The non-standard convention that relaxation times are at least 1 is made for convenience.} $C\ge 1$ such that the Poincar\'e
inequality (recall Section \ref{subsec:notation})
\begin{equation}
\label{eq:def:gammaR}
\var_{R}(f\tc\SG^\o)\le C\sum_{x\in R}\mu_{R}(c_x^{R,\omega}\var_x(f)\tc \SG^\o),
\end{equation}
holds for every $f:\O_R\to\bbR$. In the sequel we will sometimes refer to $\g^\o(R)$ as the \emph{relaxation time of $\SG^\o(R)$}. The fact that FA-$2$f restricted to $\SG^\o(R)$ is irreducible (see Remark \ref{rem:irreducibility}) implies that $\gamma^\o(R)$ is finite. However, proving a good upper bound on $\g^\o(R)$ is quite hard.
 
\begin{prop}[Relaxation time of mobile droplets]
\label{thm:droplet}
For all $n\le 2N$ 
\[\max_\o\g^\o(\L^{(n)})\leq \exp\big(O(\log^2(1/q) n )\big).\]
\end{prop}
In particular, recalling \eqref{defn}, on the final scale this yields
\[\max_\o\g^\o(\L^{(2N)})
\leq \exp\big(O(\log^3(1/q))/\sqrt q\big).\]
\begin{rem} 
We stress an important difference in the definition of $\g^\o(\L^{(n)})$
w.r.t.\ a similar definition in \cite{Hartarsky21a}*{(12)}.
Indeed, in \eqref{eq:def:gammaR} the conditioning w.r.t.\ the
super-good event $\SG^\o(R)$ appears in the l.h.s.\ \emph{and} in the r.h.s.\ of the inequality, while in 
\cite{Hartarsky21a}*{(12)} the conditioning was absent in
the r.h.s. Keeping the conditioning also in the r.h.s.\ is a delicate and important point if one wants to get a
Poincar\'e constant which is \emph{sub-leading}
w.r.t.\ $\rd^{-1}$. Theorem 4.6 of \cite{Hartarsky21a} in the
context of FA-$2$f would give a Poincar\'e constant bounded from above by $\exp(\log(1/q)^3/q)$, much bigger than $\rd^{-1}$. 
\end{rem}

\subsection{Proof of Proposition \ref{thm:droplet}}
  \label{sec:relaxation}
The proof of the constrained Poincar\'e inequality of Proposition \ref{thm:droplet} is unfortunately rather long and technical but the main idea and technical ingredients can be explained as follows.

Given the recursive definition of the super-good event  $\SG^\o(\L^{(n)})$ it is quite natural to try to bound from above its relaxation time in progressively larger and larger volumes. A high-level ``dynamical intuition'' here goes as follows. After every time interval of length $\Theta(\g^{\mathbf 1}(\L^{(n-1)}))$ the core of $\L^{(n)}$, namely a  super-good translate of $\L^{(n-1)}$ inside $\L^{(n)}$, will equilibrate under the FA-$2$f dynamics. Therefore, the relaxation time of $\SG(\L^{(n)})$ should be at most $T^{(n)}_{\rm eff}\times \g^{\mathbf 1}(\L^{(n-1)})$, where $T^{(n)}_{\rm eff}$ is the time that it takes for the core to equilibrate its position inside $\L^{(n)}$, assuming that at each time the infections inside it are at equilibrium. The main step necessary to transform this rather vague idea into a proof is as follows.

In order to analyse the characteristic time scale of the effective dynamics of a core, we need to improve and expand a well established mathematical
technique for KCM to relate the relaxation times of two
$\o$-super-good regions on different scales. Such a technique
introduces various types
of \emph{auxiliary constrained block chains} and a large part of our
argument is devoted to proving good bounds
on their relaxation times (see Section \ref{sec:Poinc}). The main application of this technique to our concrete problem is summarised in Lemmas \ref{lem:bisectio} and \ref{lem:lin}
below which easily imply Proposition \ref{thm:droplet}. Let
\[\Lambda^{(n,+)}=
\begin{cases}
R(\ell_m+1,\ell_m)& \text{if $n=2m$,}\\    
R(\ell_{m+1},\ell_m+1) &\text{if $n=2m+1$}.
\end{cases}\]
The two key steps connecting the relaxation times of super-good rectangles of increasing length scale are as follows. 

\begin{lem}[From $\ell_{\lfloor n/2\rfloor}+1$ to $\ell_{\lfloor n/2\rfloor+1}$]
\label{lem:bisectio}
For all $0\le n\leq 2N-1$ 
\[\max_\o\g^\o(\L^{(n+1)})\leq \exp(O(\log^2(q)))\max_\o \g^\o(\L^{(n,+)}).\]
\end{lem}

\begin{lem}[From $\ell_{\lfloor n/2\rfloor}$ to $\ell_{\lfloor n/2\rfloor}+1$]\label{lem:lin}
For all $0\le n\le 2N-1$
\[\max_\o \g^\o(\L^{(n,+)})\leq q^{-O(1)} \max_\o\g^\o(\L^{(n)}).\]
\end{lem}
\begin{proof}[Proof of Proposition \ref{thm:droplet}]
Lemmas \ref{lem:bisectio} and \ref{lem:lin} combined imply that 
\[
\max_\o\g^\o(\L^{(n+1)})\le \exp(O(\log(q)^2))\max_\o\g^\o(\L^{(n)}).
\]
Thus, Proposition \ref{thm:droplet} follows by induction over $n$. Indeed, $\g^\o(\L^{(0)})=1$ for all $\o\in\O_{\bbZ^2\setminus\L^{(0)}}$, since the l.h.s.\ of \eqref{eq:def:gammaR} is zero.
\end{proof}

Before proving Lemma \ref{lem:bisectio} formally, let us provide an informal description of the argument. We seek to apply a bisection technique (see \cites{Cancrini08,Hartarsky21b}) proceeding by a further induction. At each step of this bisection, we divide by two the difference of the widths (or heights) between our current rectangle (initially $\L^{(n+1)}$) and $\L^{(n,+)}$. In order to prove a recursive bound on the relaxation times $\g^\o$ of the intermediate rectangles of class $n+1$ arising in the process, we rely on Proposition \ref{key:bisection} as follows. 

We want to prove a Poincar\'e inequality on a larger rectangle, given such an inequality on a smaller one. We cover the larger one with two overlapping copies of the smaller one. We then use the relaxation in the smaller one to move the core of shape $\L^{(n)}$, witnessing it being super-good, to the intersection of the two translates. This makes the second copy super-good and allows us to resample it as well, thanks to the lower-scale Poincar\'e inequality. Thus, the events $\cF_{1,2}$ and $\cF_{2,3}$ in Proposition \ref{key:bisection} will roughly correspond to finding the core in the aforementioned overlap region (see Figure \ref{fig:partition1}).
\begin{proof}[Proof of Lemma \ref{lem:bisectio}]
Given $0\le n\le 2N-1$, let $K_n$ be the smallest integer $K>0$ such that $\lceil
(2/3)^{K} (\ell_{\lfloor n/2\rfloor+1}-\ell_{\lfloor n/2\rfloor})\rceil=1$ (if $K=0$, there is nothing to prove, since $\L^{(n,+)}=\L^{(n+1)}$). Equations \eqref{def:fn} and \eqref{defn} give
$\max_{n\le 2N-1}K_n \le O(\log(1/q))$. Consider the (exponentially increasing) sequence
\begin{equation}
d_k= \lceil
(2/3)^{K_n-k} (\ell_{\lfloor n/2\rfloor+1}-\ell_{\lfloor
n/2\rfloor})\rceil,\quad k\le K_n,
\end{equation}
and let $s_k=d_{k+1}-d_{k}$ for $k\leq K_n-1$. Next consider the collection $(R^{(k)})_{k=0}^{K_n}$ of rectangles of class $n+1$ interpolating between $\L^{(n,+)}$ and $\L^{(n+1)}$ defined by 
\[R^{(k)}=
\begin{cases}
R(\ell_m+d_k , \ell_{m}) & \text{if $n=2m$},\\
R( \ell_{m+1},\ell_m+d_k) & \text{if $ n=2m+1$.}
\end{cases}\]
By construction, $R^{(k)}\subset R^{(k+1)}$, $R^{(0)}=\Lambda^{(n,+)}$ and $R^{(K_n)}=\Lambda^{(n+1)}$. Finally, recall the events
$\SG^\o(R)$ and $\SG_s^\o(R)$ constructed in Definition \ref{def:supergood}
for any rectangle $R$ of class $n+1\le 2N$ and let
\begin{equation}
\label{eq:ak}  a_k= \max_{\o}\big( \mu_{R^{(k)}}(\SG^{\mathbf 1}_{s_k}\tc
       \SG^\o)\big)^{-2} \max_\o\big(\mu_{R^{(k)}}(\SG^\o_0 \tc \SG^\o)\big)^{-1},
\end{equation}
where $\max_\o$  is over  all $\o\in
\O_{\partial R^{(k)}}$.
In Corollary \ref{cor:T:ratio} we prove that
\begin{equation*}
\m_R\big(\SG_s^\o\tc\SG^{\o'}\big)\ge q^{O(1)}
\end{equation*}
\emph{uniformly} over all rectangles $R$ of class $n+1\le 2N$, all possible
values of the offset $s$ and all choices of the boundary configurations  $\o,\o'\in
\O_{\partial R}$. As a consequence
\begin{equation}
\label{eq:ak:bound}
  \max_{n\le 2N-1}\max_{k\le K_n} a_k\le (1/q)^{O(1)}.
\end{equation}
With the above notation the key inequality for proving Lemma \ref{lem:bisectio} is
\begin{equation}
  \label{eq:keybisection2}
 \max_{\o} \g^\o(R^{(k+1)})   \leq C
   a_{k}\times \max_{\o }\g^\o(R^{(k)}), \quad k\in [0,K_n-1],
\end{equation}
for some universal constant $C>0$.
Recalling that $R^{(0)}=\L^{(n,+)}$ and $R^{(K_n)}=\L^{(n+1)}$, from \eqref{eq:keybisection2} it follows that
\begin{equation}
\label{quote}
\max_\o \gamma^{\o}(\L^{(n+1)}) \leq \Big(C^{K_n}\prod_{k=0}^{K_n-1}a_k \Big)\times \max_\o
 \g^\o(\L^{(n,+)})
\end{equation}
which in turn implies Lemma
\ref{lem:bisectio} by \eqref{eq:ak:bound} and
$K_n \le O(\log(1/q))$.

The proof of \eqref{eq:keybisection2}, which is detailed for simplicity only in the even case $n=2m$, relies on the Poincar\'e inequality for a properly chosen \emph{auxiliary block chain} proved in Proposition \ref{key:bisection}. 
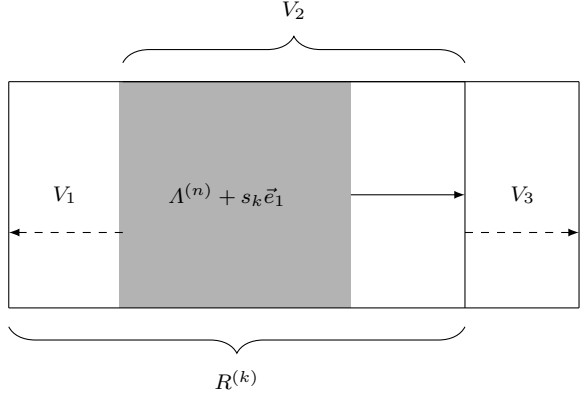
\begin{figure}\centering
\begin{tikzpicture}[>=latex,x=0.5cm,y=0.5cm]
\fill[opacity=0.3] (-0.1,6) -- (-0.1,0) -- (6,0) -- (6,6) -- cycle;
\draw (-3,0)-- (12,0);
\draw (-3,6)-- (12,6);
\draw (9,0)-- (9,6);
\draw (12,0)-- (12,6);
\draw (9,6)-- (0,6);
\draw (-3,6)-- (-3,0);
\draw [->] (6,3) -- (9,3);
\draw [->,dashed] (0,2) -- (-3,2);
\draw[anchor=east] (-1,3) node {$V_1$};
\draw[anchor=west] (10,3) node {$V_3$};
\draw[anchor=west] (1,3) node {$\L^{(n)}+s_k\vec e_1$};
\draw [decorate,decoration={brace,amplitude=10pt}]   (9,-0.5) -- (-3,-0.5) node [midway,yshift= - 0.7cm] {$R^{(k)}$};
\draw [decorate,decoration={brace,amplitude=10pt}] (0,6.5) -- (9,6.5) node [midway,yshift= 0.7cm] {$V_2$};
\draw [->,dashed] (9,2) -- (12,2);
\end{tikzpicture}
\caption{
\label{fig:partition1}
The partition of $R^{(k+1)}$ into the rectangles $V_1,V_2,V_3$. Here we
illustrate the event $\cF_{1,2}\cap\cA_3$. The grey region
$\Lambda^{(n)}+s_k\vec e_1$ at the left boundary of $V_2$  is $\SG$ and the \emph{dashed}
arrows in $V_1$ and $V_3$ indicate their $\o$-traversability. The \emph{solid} arrow in $V_2\setminus (\Lambda^{(n)}+s_k\vec e_1)$
indicates instead the $\mathbf
1$-traversability of $V_2\setminus (\Lambda^{(n)}+s_k\vec e_1)$.
Clearly the entire configuration belongs to the events $\cH$ and $\cK$
defined in \eqref{E1}, \eqref{E2}. Indeed, the two ($\o$-)right-traversability events together imply the $\o$-right-traversability of $(V_2\cup V_3)\setminus (\L^{(n)}+s_k\vec e_1)$.}
\end{figure}
In order to exploit that proposition we partition $R^{(k+1)}$ into three disjoint rectangles $V_1$, $V_2$, $V_3$ as follows (see Figure \ref{fig:partition1}):  
 \begin{align*}
   V_1&=R(s_k,\ell_m),&V_2&=R^{(k)}\setminus V_1,&V_3&=R^{(k+1)}\setminus R^{(k)}.
\end{align*}
Then, given a boundary configuration $\o\in \O_{\partial R^{(k+1)}}$, let
  \begin{align}
    \label{E1}
    \cH&=\{\eta\in\O_R^{(k+1)}: \eta_{3}\in \cT^\o_{\rightarrow}(V_3) \text{ and }
    \eta_1\cdot \eta_2\in \SG^{\eta_3\cdot\omega}(V_1\cup V_2)\},\\
\label{E2} \cK&=\{ \eta\in\O_R^{(k+1)}:\eta_{1}\in \cT^\o_{\leftarrow}(V_1) \text{ and }
       \eta_{2}\cdot\eta_{3}\in \SG^{\eta_1\cdot\omega}(V_2\cup
     V_3)\},
  \end{align}
where $\eta_i:=\eta_{V_i}$. In words, $\cH$ requires that $V_3$ is $\o$-right-traversable and $R^{(k)}=V_1\cup V_2$ is $\o\cdot \eta_{R^{(k+1)}\setminus R^{(k)}}$-super good and similarly for $\cK$. Notice that $\cH\cup \cK=\SG^\o(R^{(k+1)})$. Indeed, the width of $V_2$ is in fact $\ell_m+2d_{k}-d_{k+1}\ge\ell_m$ and therefore 
any configuration in $\SG^\o(R^{k+1})$ necessarily contains a
super-good core in either $V_1\cup V_2$ or $V_2\cup V_3$.

We next introduce two additional events (see Figure \ref{fig:partition1})
\begin{align}
  \cF_{1,2}&{}=\SG_{s_k}^{{\mathbf 1}_{V_3}\cdot \o}(V_1\cup V_2)&\cF_{2,3}&{}=\SG_{0}^{\mathbf 1_{V_1}\cdot \o}(V_2\cup V_3).\label{eq:def:F}
\end{align}
In
words, $\cF_{1,2}$ ($\cF_{2,3}$) consists of super-good configurations in $V_1\cup V_2$ ($V_2\cup V_3$) with a super-good core of type $\L^{(n)}$ inside $V_2$ in the \emph{leftmost} possible position. Monotonicity in the boundary condition easily implies that 
\[\{\eta: \eta_{3}\in \cT^\o_{\rightarrow}(V_3) \text{ and }\eta_1\cdot \eta_2\in \cF_{1,2}\}\subset \cH\cap \cK,\]
and similarly for $\cF_{2,3}$ (see Figure \ref{fig:partition1}).

We can now apply Proposition \ref{key:bisection} with
parameters $\O_i=\O_{V_i}$ for $i\in\{1,2,3\}$, $\cA_1=\cT^\o_\leftarrow(V_1)$, $\cA_3=\cT^\o_\rightarrow(V_3)$, $\cB_{1,2}^{\eta_3}=\SG^{\eta_3\cdot\o}(V_1\cup V_2)$, $\cB_{2,3}^{\eta_1}=\SG^{\eta_1\cdot\o}(V_2\cup V_3)$ and $\cF_{1,2},\cF_{2,3}$ from \eqref{eq:def:F}. We claim that
\begin{align*}
\Ta^{(1)}:={}&  \max_{\substack{\eta_{1}\in\cT_{\leftarrow}^\o(V_1)\\\eta_{3}\in \cT^\o_\rightarrow(V_3)}}\Big(\frac{\mu_{R^{(k+1)}}\big(\SG^{\eta_3\cdot\o}(V_1\cup V_2) \big)}{\mu_{R^{(k+1)}}\big(\cF_{1,2}\big)}\Big)^2
\times\frac{\mu_{R^{(k+1)}}\big(\SG^{\eta_1\cdot\o}(V_2\cup V_3) \big)}{\mu_{R^{(k+1)}}\big(\cF_{2,3}\big)}\\
={}&\max_{\eta_1,\eta_3}\m^{-2}_{R^{(k)}}\big(\SG_{s_k}^{\bone_{V_3}\cdot\o}\tc \SG^{\eta_3\cdot\o}\big)\m^{-1}_{s_k\vec e_1+R^{(k)}}\big(\SG_0^{\bone_{V_1}\cdot\o}\tc\SG^{\eta_1\cdot\o}\big)\le a_k.
\end{align*}
Indeed, the second equality follows from \eqref{eq:def:F} together with the fact that $V_1\cup V_2= R^{(k)}$ and $V_2\cup V_3=R^{(k)}+s_k$, while the inequality follows from \eqref{eq:ak}. For the inequality it suffices to use monotonicity in the boundary condition for the first term and observe that $\SG^{\eta_1\cdot\o}_0(s_k\vec e_1+R^{(k)})$ does not depend on $\eta_1$ for the second one. Thus, Proposition \ref{key:bisection} yields
\begin{multline}
\label{impo}
\var_{R^{(k+1)}}\big(f \tc \SG^\o\big)=\var_{R^{(k+1)}}\big(f \tc \cH\cup
 \cK\big)\le c a_k\\
\times\mu_{R^{(k+1)}}\big(\1_{\cH}\var_{R^{(k+1)}}(f\tc
   \cH,\eta_3) +\1_{\cK}\var_{R^{(k+1)}}(f\tc \cK,\eta_1)\tc \cH\cup \cK \big),
\end{multline}
for some universal constant $c>0$.

In order to conclude the proof of  \eqref{eq:keybisection2}
we are left with the analysis of the average
w.r.t.\ $\mu_{R^{(k+1)}}(\cdot\tc \cH\cup \cK)$ in the r.h.s.\ of
\eqref{impo}. Recalling \eqref{eq:def:gammaR} and \eqref{E1},
for any $\eta_3\in \cT^\o_{\rightarrow}(V_3)$ we get
\begin{multline}
  \label{upperV}
\var_{R^{(k+1)}}(f\tc \cH,\eta_3)\\
\begin{aligned}[t]={}&\var_{R^{(k+1)}}(f\tc\eta_3,\SG^{\eta_3\cdot\o}(R^{(k)}))\\
\leq{}&\max_{\omega'\in{\O_{\bbZ^2\setminus       R^{(k)}}}}\gamma^{\o'}(R^{(k)})\times\sum_{x\in R^{(k)}}\mu_{R^{(k)}}\big(c^{R^{(k)},\eta_3\cdot\o}_x\var_x(f)\tc\SG^{\eta\cdot\omega}\big).
\end{aligned}
\end{multline}
An analogous inequality holds for $\var(f\tc \cK,\eta_1)$ when
$\eta_1\in \cT^\o_{\leftarrow}(V_1)$.
Finally, we observe that for any $x\in R^{(k)}$
\begin{multline}
\label{eq:bisection:final}\mu_{R^{(k+1)}}\big(\1_{\cH}\mu_{R^{(k)}}\big(c_x^{R^{(k)},\eta_3\cdot\o}\var_x(f)\tc \SG^{\eta_3\cdot\omega}(R^{(k)})\big)\tc \SG^\o(R^{(k+1)})\big)\\
\begin{aligned}={}&\frac{\mu_{V_3}\big(\1_{\cA_3}\m_{R^{(k)}}\big(\1_{\SG^{\eta_3\cdot\omega}(R^{(k)})}\mu_{R^{(k)}}\big(c_x^{R^{(k)},\eta_3\cdot\o}\var_x(f)\tc \SG^{\eta_3\cdot\omega}(R^{(k)})\big)\big)\big)}{\m_{R^{(k+1)}}(\SG^\o\big)}\\
={}&\mu_{R^{(k+1)}}\big( \1_{\cH}c_x^{R^{(k+1)},\o}\var_x(f) \tc \SG^\o(R^{(k+1)})\big),
\end{aligned}
\end{multline}
since $\1_\cH=\1_{\cA_3}\1_{\SG^{\eta_3\cdot\omega}(R^{(k)})}\le\1_{\SG^\o(R^{(k+1)})}$ by \eqref{E1} and $\m_{R^{(k+1)}}=\m_{R^{(k)}}\otimes\m_{V_3}$. A similar result relation holds for $\cK$. Inserting \eqref{upperV} and \eqref{eq:bisection:final} into \eqref{impo}, we get
\begin{multline*}
\var_{R^{(k+1)}}(f \tc SG^\o) \leq
O(a_k) \times
\max_{\omega'}\gamma^{\o'}(R^{(k)}) \\
\times\sum_{x\in R^{(k+1)}}\mu_{R^{(k+1)}}\big(c_x^{R^{(k+1)},\o}\var_x(f)\tc SG^\o\big),
\end{multline*}
which proves \eqref{eq:keybisection2} in view of \eqref{eq:def:gammaR}.
\end{proof}

The proof of Lemma \ref{lem:lin} is similar to that of Lemma \ref{lem:bisectio}, but in this case we plan to use Proposition \ref{key:bisection2} instead of Proposition \ref{key:bisection}. The reason why the same proof does not apply is that the intersection of two distinct copies of $\L^{(n)}$ is never large enough to contain another copy of $\L^{(n)}$. Therefore, we are forced to look inside the $\L^{(n)}$ core in order to shrink it by one line (see Figure \ref{fig:overlineSG}). Namely, we will position the core of type $\L^{(n-2)}$ so that it is in the middle region corresponding to $V_2$ in the previous proof. We will then ask for events stronger than traversability on $V_1$ and $V_3$ in order to fit the structure in $V_2$ (see Figure \ref{fig:overlineSG}).
\begin{proof}[Proof of Lemma \ref{lem:lin}]
Once again, we provide the details only in the case $n=2m$. Let us start with the case $m=0$. Firstly, $\g^\o(\L^{(0)})=1$ for all $\o$ by the definition \eqref{eq:def:gammaR}, as $\var_{\L^{(0)}}(f\tc\SG^\o)=0$ for all $f$ and $\o$. Moreover, $\SG^\o(\L^{(0,+)})\subset\O_{\L^{(0,+)}}$ has 1, 2 or 3 elements (depending on $\o$). If this space has a single point, $\g^\o(\L^{(0,+)}=1$ as for $\L^{(0)}$ and we are done. Otherwise, we are dealing with an irreducible reversible Markov process on at most 3 states and transition rates bounded from below by $q$, so $\max_{\o}\g^{\o}(\L^{(0,+)})=q^{O(1)}$.

Let $m\geq 1$. We begin by writing
$\Lambda^{(n,+)}=R(\ell_m+1,\ell_m)=V_1\cup V_2\cup V_3$, where $V_1$ denotes 
the leftmost column, $V_3$ the rightmost column and $V_2$ all the
remaining columns (see Figure \ref{fig:overlineSG}). By construction $V_1\cup V_2$ and $V_2\cup V_3$ are translates of $\L^{(n)}$.
Then, for any given $\omega\in\Omega_{\partial  \Lambda^{(n,+)}}$,
we introduce the events
\begin{align*}
  \cM &{}=\cT^\o_{\rightarrow}(V_3)\cap \SG(V_1\cup V_2)=\SG^\o_0(\L^{(n,+)})\\
  \cN&={}\cT^\o_{\leftarrow}(V_1)\cap \SG(V_2\cup V_3)=\SG^\o_1(\L^{(n,+)})
\end{align*}
and observe that $\SG^\o(\Lambda^{(n,+)})=\cM\cup \cN$, since the only possible values of the offset $s$ in Definition \ref{def:supergood} in our case are $0$ and $1$.
In order to be able to use Proposition \ref{key:bisection2} we need
some further events. The first one is the event $\overline{\SG}(V_2)$ which is best explained by
Figure \ref{fig:overlineSG}.
It corresponds to requiring 
that inside the rectangle $V_2=R(\ell_m-1,\ell_m)+\vec e_1$ there
exists a $\mathbf 1$-super-good square $R(\ell_{m-1},\ell_{m-1})+x$ and
the remaining rectangles in $V_2\setminus R(\ell_{m-1},\ell_{m-1})+x$ which sandwich $R(\ell_{m-1},\ell_{m-1})+x$ are $\mathbf 1$-traversable. The formal definition is as follows.

\begin{defn}[Shrunken super-good]
\label{def:SGbar}
Let $R=R(\ell_m-1,\ell_m)=V_2-\vec e_1$. We say that $\overline\SG(R)$ occurs if there exist integers $0\le s_1\le \ell_m-\ell_{m-1}-1$ and $0\le s_2\le \ell_m-\ell_{m-1}$ such that the intersection  of the following events, in the sequel $\overline\SG_{s_1,s_2}(R),$ occurs (see Figure \ref{fig:overlineSG})
\begin{align*}
&\SG(\L^{(n-2)}+s_1\vec e_1+s_2\vec e_2);\\
&\cT_{\leftarrow}(R(s_1,\ell_{m-1})+s_2\vec e_2);\\
& \cT_{\rightarrow}(R(\ell_m-\ell_{m-1}-1-s_1,\ell_{m-1})+(s_1+\ell_{m-1})\vec e_1+s_2\vec e_2);\\
&\cT_{\downarrow}(R(\ell_{m}-1,s_2));\\
&\cT_{\uparrow}(R(\ell_{m}-1,\ell_m-\ell_{m-1}-s_2)+(s_2+\ell_{m-1})\vec e_2).\end{align*}
The event $\overline\SG(V_2)$ is defined by translation of $\overline\SG(R)$. Then for any $\eta_{2}\in\overline\SG(V_2)$, the segments $I_1$ and $I_3$ are given by
\begin{align*}
I_1(\eta_{2})&{}=R(1,\ell_{m-1})+s_2(\eta_{2})\vec e_2\subset V_1=R(1,\ell_m),\\
I_3(\eta_{2})&{}=R(1,\ell_{m-1})+s_2(\eta_{2})\vec e_2+\ell_m\vec e_1\subset V_3=V_1+\ell_m\vec e_1,\end{align*}
where $s_2(\eta_2)$ is an arbitrary one of the choices of $s_2$ such that $\eta_2\in\overline\SG_{s_1,s_2}(R)$ for some $s_1$.
\end{defn}

As before, let $\eta_i:=\eta_{V_i}$. Recalling Definition \ref{def:supergood} and Figures \ref{fig:super-good} and \ref{fig:overlineSG}, it is not hard to check that
\begin{equation}\label{overlineF1}
\hat\cC_{1,2}:=\big\{\eta\in\O_{V_1\cup V_2}: \eta_2\in\overline\SG(V_2),
  \eta_{I_1(\eta_2)}\neq \mathbf 1\big\}\subset\SG(V_1\cup V_2),
\end{equation}
since $I_1$ extends the horizontal traversability, while the vertical one and the core of type $\L^{(n-2)}$ are witnessed by $\overline\SG(V_2)$. For $\eta\in \hat\cC_{1,2}$ we set
\begin{equation}\label{eq:def:A3}
\cA_3^{\eta_1\cdot\eta_2}=\{\eta_{I_3(\eta_2)}\neq \mathbf 1\}.
\end{equation}
By \eqref{overlineF1} and its analogue for $I_3$ we have
\[
  \big\{\eta\in\O_{\L^{(n,+)}}:\ \eta_1\cdot\eta_2\in \hat\cC_{1,2} \text{ and } \eta_3\in
  \cA_3^{\eta_1\cdot\eta_2}\big\}\subset \cM\cap \cN.
\]

We can finally apply Proposition \ref{key:bisection2} with parameters $\O_i=\O_{V_i}$ for $i\in\{1,2,3\}$, $\cC_{1,2}=\SG(V_1\cup V_2)$, $\cC_{2,3}=\SG(V_2\cup V_3)$, $\cA_1=\cT^\o_{\leftarrow}(V_1)$, $\cA_3=\cT^\o_{\rightarrow}(V_3)$ and $\hat\cC_{1,2}$ and $\cA_3^{\eta_1\cdot \eta_2}$ as above. Set
\begin{equation}
\label{eq:def:Ta2}
\Ta^{(2)}= \max_{\eta\in \hat\cC_{1,2}}\frac{\mu_{\L^{(n,+)}}(\cA_3)}{\mu_{\L^{(n,+)}}(\cA_3^{\eta_1,\eta_2})} \times \frac{\mu_{\L^{(n,+)}}( \cC_{1,2})}{\mu_{\L^{(n,+)}} (\hat\cC_{1,2})}.
\end{equation}
Then Proposition \ref{key:bisection2} gives that for some $c>0$ we have
  \begin{multline}
\var_{\L^{(n,+)}}(f\tc\SG^\o)=\var_{\L^{(n,+)}}(f\tc \cM\cup \cN)\\
\le c \Ta^{(2)}\times \mu_{\L^{(n,+)}}\Big(\1_{ \cM}\var_{\L^{(n,+)}}(f\tc \cC_{1,2},\eta_3)+\1_{ \cN}\var_{\L^{(n,+)}}(f\tc \cC_{2,3},\eta_1)\\
  +\1_{ \cM}\var_{\L^{(n,+)}}(f\tc \cA_3,\eta_1,\eta_2) + \1_{ \cN}\var_{\L^{(n,+)}}(f\tc \cA_1,\eta_2,\eta_3)\tc \cM\cup \cN\Big).
   \label{impo3}
 \end{multline}
By \eqref{eq:def:A3}, $\min_{\eta\in
  \hat\cC_{1,2}}\mu_{\L^{(n,+)}}(\cA_3^{\eta_1\cdot\eta_2})\ge q$.
Furthermore, in Lemma \ref{lem:SGbar} we will establish that
$\mu_{V_1\cup V_2}(\hat\cC_{1,2}\tc \cC_{1,2})\ge
q^{O(1)}$. Combining these observations with \eqref{eq:def:Ta2}, we get
\begin{equation}
  \label{ubound2}\Ta^{(2)}\leq q^{-O(1)}. 
\end{equation}

\begin{figure}\centering
\begin{tikzpicture}[>=latex,x=0.5cm,y=0.5cm]
\draw [fill=black,fill opacity=0.5] (3,3) rectangle (8,8);
\draw [line width=1pt] (-1,0) rectangle (10,10); 
\draw [pattern=north east lines, opacity=0.7] (0,3) rectangle (3,8);
\draw [->,thick] (3,5.5) -- (0,5.5);
\draw [pattern=north east lines, opacity=0.7] (8,3) rectangle (9,8);
\draw [->,thick] (8,5.5) -- (9,5.5);

\draw [pattern=dots, opacity=0.5] (0,8) rectangle (9,10);
\draw [->,thick] (5.5,8) -- (5.5,10);
\draw [pattern=dots, opacity=0.5] (0,0) rectangle (9,3);
\draw [->,thick] (5.5,3) -- (5.5,0);

\draw[anchor=east] (-1,2) node {$V_1$};
\draw [decorate,decoration={brace,amplitude=10pt}] (0,10.2) -- (9,10.2) node [midway,yshift=0.6cm] {$ V_2$};
\draw [decorate,decoration={brace,amplitude=10pt}] (10.2,8) --
(10.2,3) node [midway,xshift=0.6cm] {$ I_3$};
\draw [decorate,decoration={brace,amplitude=10pt}] (-1.2,3) -- (-1.2,8) node [midway,xshift=-0.6cm] {$ I_1$};
\draw[anchor=west] (10,2) node {$V_3$};
\end{tikzpicture}
\caption{
\label{fig:overlineSG}
The partition of $\Lambda^{(n,+)}$ into the rectangle $V_2$ and the two columns $V_1$ and $V_3$. Here we illustrate the event $\overline{\SG}(V_2)$: the grey region is a super-good rectangle of the type $\L^{(n-2)}$, while the patterned rectangles are $\mathbf 1$-traversable in the arrow directions.  If there is at least one infection in $I_3$ then the rectangle $V_2\cup V_3$ is super-good. Similarly, an infection in $I_1$ suffices to make $V_1\cup V_2$ super good.}
\end{figure}
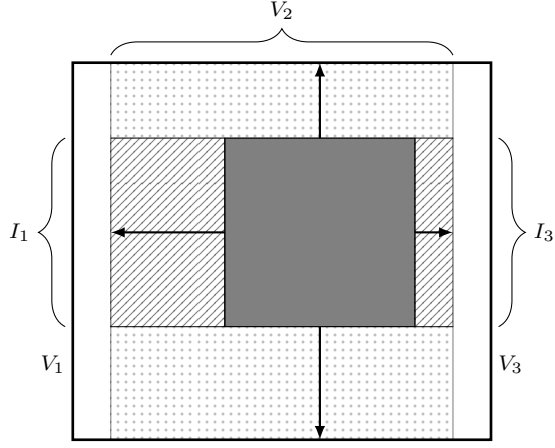
We now turn to examine the four averages w.r.t.\ $\mu_{\L^{(n,+)}}(\cdot \tc \cM\cup \cN)$ appearing in the r.h.s.\ of \eqref{impo3}. Recall that
$\cM\cup \cN=\SG^\o(\L^{(n,+)} )$. Proceeding as for the r.h.s.\ of \eqref{impo}, we obtain that
\begin{multline}
\mu_{\L^{(n,+)}}\big(\1_{ \cM}\var_{\L^{(n,+)}}(f\tc \cC_{1,2},\eta_3)+\1_{ \cN}\var_{\L^{(n,+)}}(f\tc \cC_{2,3},\eta_1)\tc\cM\cup\cN\big)\\
\le O(\gamma^{\bone}(\L^{(n)}))\times \sum_{x\in \L^{(n,+)}}\mu_{\Lambda^{(n,+)}}\big(c_x^{\L^{(n,+)},\bone}\var_x(f)\tc \SG^\o\big).
\label{eq:12}
\end{multline}
Indeed, the only difference is that $\cC_{1,2}=\SG(\L^{(n)})$, so that we recover a $\bone$ boundary condition, and we use that $c_x^{\L^{(n)},\bone}\le c_x^{\L^{(n,+)},\bone}$ and similarly for $V_2\cup V_3$ instead of $\L^{(n)}$. We will now explain how to upper bound the third average in \eqref{impo3},
\begin{equation}\label{eq:third average1}
  \mu_{\L^{(n,+)}}\Big(\1_{ \cM}\var_{\L^{(n,+)}}(f\tc \cA_3,\eta_1,\eta_2) \tc \cM\cup \cN\Big),
\end{equation} the fourth one being similar.
We need to distinguish
two cases, according to whether the boundary condition $\o$ has an
infection on the column $V_3+\vec e_1$ or not.
\paragraph{Assume $\o_{V_3+\vec e_1}=\mathbf 1$.} In this case $\cA_3=\cT_{\rightarrow}(V_3)=\O_{V_3}\setminus\{\mathbf 1\}$ and 
Proposition \ref{prop:FA1f}\ref{item:FA1ferg}, gives that
\begin{align}
\var_{\L^{(n,+)}}(f\tc \cA_3,\eta_1,\eta_2)&{}=  \var_{V_3}(f\tc
\cT_{\rightarrow}(V_3))\nonumber\\
&{}\leq q^{-O(1)}\sum_{x\in
    V_3}\mu_{V_3}\big(\tilde c_x\var_x(f)\tc \cT_{\rightarrow}(V_3)\big),
  \label{condition}
\end{align}
with $\tilde c_x(\eta)=1$ if $x$ has at least one infected neighbour
inside $V_3$ and $\tilde c_x(\eta)=0$ otherwise. For $x\in V_3$ let
\begin{equation}
\label{eq:def:Ax}
  A_x= \mu_{\L^{(n,+)}}\Big(\1_{ \cM}\ \mu_{V_3}\big(\tilde
  c_x\var_x(f)\tc\cT_{\rightarrow}(V_3)\big)\tc \cM\cup \cN\Big).
\end{equation}
  
Recall that $\SG^\o(\L^{(n,+)})=\cM\cup\cN\supset\cM=\SG(\L^{(n)})\cap\cT_{\rightarrow}(V_3)$ and $\m_{\L^{(n,+)}}=\m_{\L^{(n)}}\otimes\m_{V_3}$. Then we have
\begin{align*}
A_x={}&\frac{\m_{\L^{(n,+)}}(\cM)}{\m_{\L^{(n,+)}}(\SG^\o)}\m_{\L^{(n,+)}}\big(\tilde c_x\var_x(f)\tc\cM\big)\\
={}&\frac{\m_{\L^{(n,+)}}(\cM)}{\m_{\L^{(n,+)}}(\SG^\o)}
\m_{\L^{(n,+)}}\big(\tilde c_x\m_{\L^{(n)}}\big(\var_x(f)\tc\SG(\L^{(n)})\big)\tc\cM\big)\\
\le{}&q^{-1}\m_{V_3}\big(\m_{\L^{(n)}}\big(\var_x(f)\tc\SG(\L^{(n)})\big)\tc \tilde c_x=1\big),
\end{align*}
the inequality using $\cM\subset\SG^\o(\L^{(n,+)})$, the fact that $\tilde c_x=1$ implies $\cT_\rightarrow (V_3)$ and $\m(\tilde c_x\tc \cT_\rightarrow(V_3))\ge q$ (here we use that $V_3$ is not a singleton, which follows from $m\ge 1$). Then, by the law of total variance, we get
\begin{equation}
\label{eq:Ax:bound}
A_x \le q^{-1}\mu_{V_3}\big(\var_{\{x\}\cup \L^{(n)}}(f\tc \SG( \L^{(n)}))\tc \tilde c_x=1\big).
\end{equation}
Next, we use Proposition \ref{prop:2block} with parameters $\bbP=\m_{\{x\}\cup\L^{(n)}}(\cdot\tc\SG(\L^{(n)}))$, $X_1=\eta_{\L^{(n)}}$, $X_2=\eta_{x}$, $\cH=\{\h\in\O_{\L^{(n)}}:\h_{x-\vec e_1}=1\}$, in order to write
\begin{multline}
\label{eq:lin:varxL}\var_{\{x\}\cup  \L^{(n)}}(f \tc \SG( \L^{(n)})) \leq
\frac{2}{q}\m_{\{x\}\cup\L^{(n)}}\big(\var_{ \L^{(n)}}(f \tc \SG( \L^{(n)}))\\
+\1_{\{\eta_{x-\vec e_1}=0\}}\var_x(f)\tc\SG(\L^{(n)})\big).
\end{multline}
Recalling \eqref{eq:def:gammaR}, we get
\begin{align}
\nonumber\var_{\L^{(n)}}(f \tc \SG)\le{}&
\gamma^{\mathbf 1}(\L^{(n)})\sum_{y\in \L^{(n)}}\mu_{\L^{(n)}}\big(c_y^{\L^{(n)},\mathbf 1}\var_y(f)\tc \SG
\big)\\
\le{}&\g^{\mathbf 1}(\L^{(n)})\sum_{y\in \L^{(n)}}\mu_{\L^{(n)}}\big(c_y^{\L^{(n,+),\o}}\var_y(f)\tc \SG\big),
\label{eq:varLn}\end{align}
because $c_y^{\L^{(n)},\mathbf 1}\leq c_y^{\Lambda^{(n,+)},\o}$ for any $y\in\L^{(n)}$ and $\o\in\O_{\bbZ^2\setminus \L^{(n,+)}}$. Finally,
observe that $\1_{\{\eta_{x-\vec e_1}=0\}}\tilde c_x\leq c_x^{\Lambda^{(n,+)},\o}$, because if
$x\in V_3$ has an infected neighbour in $V_3$ (the constraint $\tilde
c_x$) and $x-\vec e_1\in V_2$ is also infected, then $x$ has two infected
neighbours in  $\Lambda^{(n,+)}$. Thus, putting \eqref{eq:lin:varxL} and \eqref{eq:varLn} together, we obtain
\begin{multline}
\label{eq:cxtilde}
\tilde c_x\var_{\{x\}\cup\L^{(n)}}(f\tc\SG(\L^{(n)}))\le\frac{2}{q}\g^\bone(\L^{(n)})\times\\
\sum_{y\in\{x\}\cup\L^{(n)}}\m_{\{x\}\cup\L^{(n)}}\big(c_y^{\L^{(n,+)},\o}\var_y(f)\tc\SG(\L^{(n)})\big).
\end{multline}
Combining \eqref{condition}-\eqref{eq:Ax:bound} and \eqref{eq:cxtilde}, yields 
\begin{multline}
  \mu_{\L^{(n,+)}}\Big(\1_{ \cM}\var(f\tc \cA_3,\eta_1,\eta_2) \tc
    \cM\cup \cN\Big) \le q^{-O(1)}\sum_{x\in V_3} A_x\\
    \le \frac{\g^{\mathbf 1}(\L^{(n)})}{q^{O(1)}}\sum_{x\in V_3}\sum_{y\in \Lambda^{(n,+)}}
    \mu_{\L^{(n,+)}}\Big(c_y^{\Lambda^{(n,+)},\o}\var_y(f)\tc \SG(\Lambda^{(n)}),\tilde c_x=1\Big).
\label{eq:third average} \end{multline}
Moveover, $\SG(\L^{(n)})$ and $\tilde c_x=1$ imply $\SG^\o(\L^{(n,+)})$ and 
\[\frac{\m_{\L^{(n,+)}}(\SG^\o)}{\m_{\L^{(n)}}(\SG)\m_{V_3}(\tilde c_x)}\le\frac{2}{\m_{V_3}(\tilde c_x)}\le 2/q,\]
since there are only two possible positions for the core of type $\L^{(n)}$ of $\SG^\o(\L^{(n,+)})$. Thus, \eqref{eq:third average} is at most
\[\frac{\g^{\mathbf 1}(\L^{(n)})|V_3|}{q^{O(1)}}\sum_{y\in \Lambda^{(n,+)}}\mu_{\L^{(n,+)}}\Big(c_y^{\Lambda^{(n,+)},\o}\var_y(f)\tc \SG^\o(\Lambda^{(n,+)})\Big).\]
Moreover, $|V_3|=\ell_m\le \ell_N=q^{-O(1)}$ by \eqref{defellN}, so, recalling \eqref{eq:def:gammaR}, we are done with the case $\o_{V_3+\vec e_1}=\mathbf 1$.

\paragraph{Assume $\o_{V_3+\vec e_1}\neq \mathbf 1$.}
In this case 
$\cT^\o_{\rightarrow}(V_3)=\Omega_{V_3}$, so that $\var_{V_3}(f\tc
\cT^\o_{\rightarrow}(V_3))=\var_{V_3}(f)$.
The proof is then identical to the previous one except for inequality
\eqref{condition} which now follows from Proposition \ref{prop:FA1f}\ref{item:FA1fboundary} with the unconstrained site $z\in V_3$ chosen arbitrarily so that $\o_{z+\vec e_1}=0$.
\end{proof}

\section{Proof of Theorem \ref{th:FA2f}: upper bound}
\label{sec:upperbound}
As already announced we will only discuss the two dimensional case. The starting point is
as in \cite{Hartarsky21a}*{Section 5}. Let $\kappa$ be a large enough constant, let
\begin{equation}
\label{eq:def:tstar}
t_*= \exp\Big(
  \frac{\pi^2}{9q}\big(1+\k\sqrt{q}\log^{3}(1/q)\big)\Big)
\end{equation}
and let $T=\lfloor\exp(\log^4(1/q)/q)\rfloor$. Then
\begin{align*}
\bbE_\mu(\t_0)={}&\int_0^{+\infty}ds\,\bbP_\mu(\t_0>s)\\
={}&\int_{0}^{t_*}ds\,\bbP_\mu(\t_0>s) + \int_{t_*}^T ds\,\bbP_\mu(\t_0>s)+\int_T^{+\infty} ds \,\bbP_\mu(\t_0>s) \\
\le{}& t_* + T\bbP_\mu(\t_0>t_*)  +\int_T^{+\infty} ds \,\bbP_\mu(\t_0>s).
\end{align*}
The term $t_*$ has exactly the form required in \eqref{eq:FA2f:upper}. In order to bound the last term in the r.h.s.\ above, we use that
$\bbP_\mu(\t_0>s)\le e^{-s q/\trel}$ for all $s>0$  (see \emph{e.g.}\ \cite{Cancrini09}*{proof of Theorem 4.7}) together with $\trel\le e^{O(|\log(q)|^3/q)}$ (see \cite{Martinelli19a}*{Theorem 2 (b)}) to get that
\[
\lim_{q\to 0} 
\int_T^{+\infty} ds \,\bbP_\mu(\t_0>s)\le \lim_{q\to 0}
\frac{\trel}{q}e^{-q T/\trel}=0.
\]
In conclusion, the proof of \eqref{eq:FA2f:upper} for $\bbE_\mu(\t_0)$ boils down to proving
\begin{equation}
  \label{eq:caldo}
  \lim_{q\to 0}T\bbP_\mu(\t_0>t_*)=0.  
\end{equation}
Similarly, \eqref{eq:FA2f:upper} for $\t$ w.h.p.\ follows, since \eqref{eq:caldo} gives
\[\bbP_\m(\t_0>t_*)\le o(1/T)\le o(1).\]

The key ingredients to prove \eqref{eq:caldo} are Propositions \ref{prop:probadroplet} and \ref{thm:droplet} and Proposition \ref{prop:g-CBSEP} below. The latter is a Poincar\'e inequality for an auxiliary 
process, the \emph{generalised coalescing and branching symmetric exclusion process} ($g$-CBSEP), preliminarily studied in \cite{Hartarsky22CBSEP}.
Once we have these key ingredients, the
strategy to prove \eqref{eq:caldo} is similar to the one in \cite{Hartarsky21a}*{Section 5}. In particular, for the first part of the proof (Section \ref{sec:2steps}) we will omit most of the details and refer to \cite{Hartarsky21a}*{Section 5} for a more detailed explanation. 

\subsection{The $g$-CBSEP process}
\label{sec:CBSEP}
Given a finite connected graph $G=(V,E)$ and a finite probability space $(\cS,\p)$, assign a variable $\s_x\in \cS$ to each vertex $x\in V$ and
write $\s=(\s_x)_{x\in V}$ and $\p_G(\s)=\prod_x\p(\s_x)$. Fix
also a
bipartition $\cS_1\sqcup \cS_0=\cS$ such that $\p(\cS_1)>0$ and define the  projection $\f:\cS^V\to \{0,1\}^V$ by $\varphi(\s)=(\1_{\{\s_x\in
  \cS_1\}})_{x\in V}$. We will say that a vertex $x$ is occupied by a
\emph{particle} if $\s_x\in \cS_1$ and we will write $\O_G^+\subset
\O_G=\cS^V$ for the set of configurations $\s$ with at least one particle. Finally, for any edge $e=\{x,y\}\in E$ let $\cE_e$
be the event that there exists a particle at $x$ or at $y$.

The $g$-CBSEP continuous time Markov chain on $\O^+_G$ with parameters
$(\cS,\cS_1,\p)$ runs as
follows. The state $\{\s_x,\s_y\} $ of every edge $e=\{x,y\}$ for
which $\cE_e$ holds is resampled
with rate one (independently of all the other edges) w.r.t.\ $\p_x\otimes\p_y(\cdot\tc \cE_e)$. Thus, an edge containing
exactly one particle can swap the position of the particle between its
endpoints or can create a new particle at the empty endpoint (a
branching transition). An edge
with two particles can kill one of them (a coalescing transition) with equal probability or keep
them untouched. Notice also that the state of an edge can change completely even when the particles are untouched.
\begin{rem}
\label{rem:CBSEP}
When the parameters $(\cS,\cS_1,\pi)$ are the two point space $\{0,1\},$ the set $\{1\}$ and the Bernoulli(p) measure on $\cS$ respectively, the $g$-CBSEP chain is called the CBSEP chain on $G$ with parameter $p$. It is easy to verify that the projection of the $g$-CBSEP chain under the mapping $\varphi$ defined above coincides with the CBSEP chain with parameter $p=\pi(\cS_1)$. This observation will be used in the proof of  Proposition \ref{prop:g-CBSEP} below. We will use $g$-CBSEP rather than plain CBSEP, because the space $\cS$ will correspond to the state of the chain in a mesoscopic box. The event $\cS_1$ will correspond to the presence of a mobile droplet in this box. 
\end{rem}

It is immediate to check that $g$-CBSEP is ergodic on
$\O_G^+$ with reversible stationary measure $\p_G^+:=\p_G(\cdot\tc
\O_G^+)$ and that its Dirichlet form $\cD^{\text{$g$-CBSEP}}(f)$ for $f:\O_G^+\to \bbR$,
takes the form
\[
  \cD^{\text{$g$-CBSEP}}(f)= \sum_{e\in
    E}\p^+_G\big(\1_{\cE_e}\var_{e}(f\tc \cE_e)\big),
\]
where $\var_{e}(f\tc \cE_e)$ is the variance w.r.t.\ $\s_x,\s_y$ conditioned on $\cE_e$ if $e=\{x,y\}$.
Let now $\trel[\text{$g$-CBSEP}]$ be the \emph{relaxation time} of $g$-CBSEP
on $\O_G^+$ defined as the best constant $C$ in the Poincar\'e
inequality
\[
  \var_{\p_G^+}(f)\le C\cD^{\text{$g$-CBSEP}}(f).
\]
In the above setting the main result needed to prove \eqref{eq:caldo}
is as follows. For any positive integers $d$ and $L$ set $n=L^d$ and let $\bbZ_L=\{0,1,\dots,
L-1\}$ be the set of remainders modulo $L$. The $d$-dimensional
discrete torus with $n$ vertices, $\bbT^d_n$ in the sequel, is the set $\bbZ_L^d$
endowed with the graph structure inherited from $\bbZ^d$. In what follows we will allow $\cS$, $\cS_1$ and $\p$ to depend on $n$.
\begin{prop}
  \label{prop:g-CBSEP}
Let $d\ge 2$, $G=\bbT^d_n$ and assume that $\lim_{n\to
  \infty}n \p(\cS_1)=+\infty$.
Then, as $n\to \infty$,
for any function $f:\O_G^+\mapsto \bbR$
\[
 \var_{\p_G^+}(f)\le O\big(\p(\cS_1)^{-1}\max\big(1,\log\big(\p(\cS_1)^{-1}\big)\big)\big)\cdot \cD^{\text{$g$-CBSEP}}(f).
\]
In particular,
\[
\trel[\text{$g$-CBSEP}]\le
O\big(\p(\cS_1)^{-1}\max\big(1,\log\big(\p(\cS_1)^{-1}\big)\big)\big).
\]
\end{prop}
This is proved in Appendix \ref{app:CBSEP}.

\subsection{Transforming \eqref{eq:caldo} into a Poincar\'e inequality}
\label{sec:2steps}
Using standard ``finite speed of propagation'' bounds (see \cite{Hartarsky21a}*{Section 5.2.1}), \eqref{eq:caldo} follows if we prove
\begin{equation}
\label{eq:caldo2}
\lim_{q\to 0}T\bbP_\mu\big(\t_0^{(n)}>t_*\big)=0
\end{equation}
where $\t_0^{(n)}$ is the infection time of the origin $\t_0$ for FA-2f on the discrete torus $\bbT_n^2$ with linear size $\sqrt{n}=2T$. For this purpose  we fix a small positive constant $\d<1/2$ and choose $N_\d=N-\lfloor\log(1/\d)/\sqrt{q}\rfloor$ where $N=\lceil\frac{8\log(1/q)}{\sqrt{q}}\rceil$ is the final scale in the droplet construction  (see \eqref{defn}). With this choice $\ell_{N_\d}\simeq \d \ell_N=\d/q^{17/2 +o(1)}$ (cf.\ \eqref{def:fn}) and w.l.o.g.\ we assume that $\ell_{N_\d}$ divides $2T$.
We then partition the torus $\bbT_n^2$ into $M=n/\ell^2_{N_\d}$ equal mesoscopic disjoint boxes $(Q_j)_{j=1}^M$, where each $Q_j$ is a suitable lattice translation by a vector in $\bbT_n^2$ of the box $Q=[\ell_{N_\d}]^{2}=\Lambda^{(2N_{\delta})}$ (see \eqref{def:Lambdan}). The labels of the boxes can be thought of as belonging to the new torus $\bbT_M^2$ and we assume that $Q_i,Q_j$ are neighbouring boxes in $\bbT_n^2$ iff $i, j$ are neighbouring sites in $\bbT_M^2$. In $\O_{\bbT_n^2}$  we consider the event
\begin{equation}
    \label{eq:def:E}
      \cE=\bigcup_{j\in \bbT_M^2} \cS\cG_{j}\cap \bigcap_{i\in \bbT_M^2} \cG_i
\end{equation}
where $\SG_j$ is the event that $Q_j$ is super-good (see Definition \ref{def:droplets}) and $\cG_i$ is the event that any row and any column (of lattice sites) of $Q_i$ contains an infected site.

In order to apply the same strategy as \cite{Hartarsky21a}*{Section 5}
it is crucial to have that the ``environment'' characterised by $\cE$
is so likely that (cf.\ \cite{Hartarsky21a}*{(28)})
\begin{equation}
  \label{eq:scale choice}
\lim_{q\to 0}\mu(\cE^c)T^{3} t_*=0.  
\end{equation}
Using  $t_*=e^{\frac{\pi^2}{9q}(1+o(1))}, T=\lfloor\exp(\log^4(1/q)/q)\rfloor, M=4T^2/\ell^2_{N_{\d}}$, $\ell_{N_\d}=O(1/q^9),$ together with  Proposition \ref{prop:probadroplet}, it follows that
\[
\lim_{q\to 0} T^3t_*\mu\big(\bigcap_{j\in \bbT_M^2} \cS\cG^c_{j}\big)\le \lim_{q\to 0} T^3t_* \big(1-e^{-\frac{\pi^2}{9q}(1+O(\sqrt{q}\log^{2}(1/q)))}\big)^{M} =0.
\]
Similarly, using $\mu(\cG_i^c)\le 2\ell_{N_\d} (1-q)^{\ell_{N_\d}}\le 2\ell_{N_\d} e^{-q\ell_{N_\d}}$ together with $\ell_{N_\d}=\O( 1/q^{8})$, we get
\[
\lim_{q\to 0} T^3t_*\mu\big(\bigcup_{i\in \bbT_M^2}\cG^c_{i}\big)\le T^3t_* M2\ell_{N_\d}e^{-q\ell_{N_\d}} =0,
\]
and \eqref{eq:scale choice} follows.

An easy consequence of \eqref{eq:scale choice} (cf.\ \cite{Hartarsky21a}*{Eq.~(29)}) is that as $q\to 0$
\begin{equation}
\label{eq:T:environment}
  T\bbP_\mu\big(\t_0^{(n)}\ge t_*\big)\le T\bbP_\mu\big(\t^{(n)}_\cF\ge t_*\big) +o(1),
\end{equation}
where $\t^{(n)}_\cF$ denotes the hitting time of the set $\cF=\{\o:\o_0=0\}\cup \cE^c$ for the FA-2f chain on $\bbT_{n}^2$. In order to bound from above the term $\bbP_\mu\big(\t^{(n)}_\cF\ge t_*\big)$ we follow the
standard ``variational'' approach (see \cite{Hartarsky21a}*{Eq.~(30)} and \cite{Asselah01}*{Theorem 2}). 

Let $\cD_{\bbT_n^2}(f)=\sum_{x\in \bbT_n^2}\mu_{\bbT_n^2}\big(c_x^{\bbT_n^2}\var_x(f)\big),$ where $c_x^{\bbT_n^2}$ is the FA-2f constraint at $x$ for the torus $\bbT_n^2$ (see \eqref{def:ratesbis}), be 
the Dirichlet form of the FA-2f chain on the torus $\bbT_n^2$. Then
\begin{align}
\label{eq:lF}
T\bbP_\mu\big(\t^{(n)}_\cF\ge t_*\big)\le{}& Te^{-t_* \l^{(n)}_\cF},&\l^{(n)}_\cF={}& \inf\Big\{  \frac{\cD_{\bbT_n^2}(f)}{\mu_{\bbT_n^2}(f^2)}:f|_{\cF}=0\Big\}.
\end{align}
It remains to prove a precise lower bound on the coefficient $ \l^{(n)}_\cF$.

\subsection{Bounding $\l^{(n)}_\cF$ from below}
\label{sec:lambdaF} 
The last and most important step is to prove that
\begin{equation}
  \label{eq:lambdaF}
 \lambda^{(n)}_\cF\ge    e^{-O(\log^{3}(1/q)/\sqrt{q})} \rd \ge e^{-\frac{\pi^2}{9q}(1+O(\sqrt q\log^3(1/q)))},
\end{equation}
where $\rd\ge \exp(-\frac{\pi^2}{9q}(1+O(\sqrt{q}\log^{2}(1/q))))$ is the probability
that a box $[\ell_N]^2$ is super-good  (cf.\ Proposition \ref{prop:probadroplet}). Once
\eqref{eq:lambdaF} is established, recalling \eqref{eq:T:environment} and \eqref{eq:lF}, the proof of \eqref{eq:caldo2} is
complete because $t_*\l^{(n)}_\cF$ diverges rapidly enough as $q\to 0$ if the constant $\kappa$ in the definition \eqref{eq:def:tstar} of $t_*$ is chosen large enough.

It was proved in \cite{Hartarsky21a}*{Eq.~(31)}
that $\l^{(n)}_\cF\ge q\inf_{f}\frac{\cD_{\bbT_n^2}(f)}{\var_{\bbT_n^2}(f\tc \cE)}$, where the infimum is over $f:\O_{\bbT_n^2}\to\bbR$ such that $f|_{\cE^c}=0$ and $f|_\cE$ is not constant. In what follows $f$ will denote an arbitrary such function and the various constants involved in the estimates will be uniform in $f$. Hence, \eqref{eq:lambdaF} follows, once we prove that
\begin{equation}
  \label{eq:lambdaF2}
\frac{\cD_{\bbT_n^2}(f)}{\var_{\bbT_n^2}(f\tc \cE)} \ge    \exp\Big(-\frac{\pi^2}{9q}\big(1+O(\log^{3}(1/q)\sqrt{q})\big)\Big).
\end{equation}
\paragraph{Bounding $\var_{\bbT_n^2}(f\tc \cE)$, using  Proposition \ref{prop:g-CBSEP}.} Write $G$ for the graph
$\bbT_M^2$, $\cS$ for the state space
$\cG_i\subset\O_{Q_i}$ with $i\in\bbT_M^2$, $\p$ for $\mu_{Q_i}(\cdot\tc \cG_i)$ and $\cS_1\subset \cS$ for the
event $\SG_i$. Since both $\cG_i$ and $\cS\cG_i$ are increasing in the set of infections, 
\begin{equation}
\label{eq:piS1}
\pi(\cS_1)\ge \mu(\cS\cG_i)\ge \exp\Big(-\frac{\pi^2}{9q}\big(1+O\big(\sqrt{q}\log^{2}(1/q)\big)\big)\Big),
\end{equation}
where we used the Harris inequality \cite{Harris60} for the first inequality and Proposition
\ref{prop:probadroplet} for the second one.
Recalling that $M=n/\ell_{N_{\delta}}^2$ with $n=4T^2$, $\ell_{N_\d}=O(1/q^9)$ and $T=\lfloor\exp(\log^4(1/q)/q)\rfloor$, the above bound implies $\lim_{q\to 0}M\pi(S_1)=+\infty$ so that the requirement of Proposition
\ref{prop:g-CBSEP} is fulfilled.

With this notation we consider the $g$-CBSEP on $\O_G^+$ with parameters $(\cS,\cS_1,\p)$. Recalling \eqref{eq:def:E}, we identify $f$ with a function $f_G:\O_G^+\to \bbR$ via the  natural bijection between $\cE$ and $\O_G^+$: $f(\o) = f_G(\o_{Q_1},\dots,\o_{Q_{M}})$. Under this bijection
\begin{align*}
\var_{\p^+_G}(f_G)&{}=\var_{\bbT_n^2}(f\tc \cE),\\
\cD^{\text{$g$-CBSEP}}(f_G)&{}=\sum_{i\sim j}\mu_{\bbT_n^2}\big(\1_{\SG_{i,j}}
\var_{Q_i\cup Q_j}(f\tc\SG_{i,j})\tc \cE\big),
\end{align*}
where $\cS\cG_{i,j}$ is a shorthand notation for the event $(\SG_i\cup \SG_j)\cap\cG_i\cap\cG_j$ and $\sum_{i\sim j}$ denotes the sum over pairs, each counted once, of adjacent boxes. Using Proposition
  \ref{prop:g-CBSEP} and \eqref{eq:piS1} we conclude that 
\begin{align}
    \var_{\bbT_n^2}(f \tc \cE)={}&\var_{\p^+_G}(f_G)
    \le O(\p(\cS_1)^{-1}\log(1/\p(\cS_1))\big)\cD^{\text{$g$-CBSEP}}(f_G)\nonumber\\
   \le{}&
     \label{eq:part1}\begin{multlined}[t]
   \exp\Big(\frac{\pi^2}{9q}(1+O(\sqrt{q}\log^{2}(1/q)))\Big)
  \\\times\sum_{i\sim j}\mu_{\bbT_n^2}\big(
    \1_{\SG_{i,j}}
     \var_{Q_i\cup Q_j}(f\tc \SG_{i,j})\tc \cE\big).\end{multlined}
  \end{align}
  
\paragraph{Bounding $\cD_{\bbT_n^2}(f)$, using Proposition \ref{thm:droplet}.}
We next compare the sum appearing in the
r.h.s.\ of \eqref{eq:part1} to the Dirichlet form $\cD_{\bbT_n^2}(f)$ and prove that the ``comparison cost'' is at most $\exp\big(O\big(\log^3(1/q)/\sqrt q\big)\big)$, so sub-leading w.r.t.\ the main term $\exp(\frac{\pi^2}{9q})$ in \eqref{eq:part1}.
  \begin{lem}
\label{lem:comparison}
\[
\sum_{i\sim j}\mu_{\bbT_n^2}\big(\1_{\SG_{i,j}}
     \var_{Q_i\cup Q_j}(f\tc\SG_{i,j})\tc
     \cE\big)
     \le\\
      e^{O(\log^{3}(1/q)/\sqrt{q})}\cD_{\bbT_n^2}(f).
\] 
\end{lem}
  \begin{rem}
As it will be clear from the proof, we actually prove a stronger statement, namely the constraint $c_x^{\bbT_n^2}$ in the expression of $\cD_{\bbT_n^2}(f)$ will appear multiplied by the indicator that $x$ belongs to a droplet. While for many choices of $f$ the presence of this additional constraint may completely change the average $\mu_{\bbT_n^2}\big(c_x^{\bbT_n^2}\var_x(f)\big)$, it is possible to exhibit choices of $f$, for which $$\1_{\{\text{$x$ belongs to a
       ``droplet''}\}}c_x^{\bbT_n^2}\var_x(f)\simeq c_x^{\bbT_n^2}\var_x(f).$$
    \end{rem}
Before proving Lemma \ref{lem:comparison}, let us observe that, together with \eqref{eq:part1}, it implies the desired \eqref{eq:lambdaF2}. Lemma \ref{lem:comparison} itself follows by summing the bound from Claim \ref{claim:pippo} below.

\begin{clm}
  \label{claim:pippo}
Fix two adjacent boxes $Q_i,Q_j$ and let $\L_{i,j}\supset Q_i\cup Q_j$ be a translate of the box $\L^{(2N)}$. Then
\begin{multline*}\mu_{\bbT_n^2}\big(
    \1_{\SG_{i,j}}
    \var_{Q_i\cup Q_j}(f\tc \SG_{i,j})\tc \cE\big)\\
    \le e^{O(\log^{3}(1/q)/\sqrt{q})}\sum_{x\in \L_{i,j}}\mu_{\bbT_n^2}\big(\1_{\cS\cG(\L_{i,j})}c_x^{\bbT_n^2}\var_x(f)\big). 
\end{multline*}
\end{clm}
\begin{proof}[Proof of Claim \ref{claim:pippo}]
Let $\cG=\bigcap_{k\in \bbT_M^2}\cG_k \supset \cE$ and recall that $\mu(\cE)=1-o(1)$. Let $\rho_{i,j}=\mu(\cS\cG_{i,j}\tc \cG
)$ and observe that the term $\var_{Q_i\cup Q_j}(\cdot)$
does not depend on the variables $\o_{Q_i},\o_{Q_j}$. Thus, 
\begin{multline}
\mu_{\bbT_n^2}\big(\1_{\SG_{i,j}}\var_{Q_i\cup Q_j}(f\tc \SG_{i,j})\tc\cE\big)\\
\begin{aligned}[t]
\le{}& (1+o(1))
\mu_{\bbT_n^2}\big(\1_{\SG_{i,j}}\var_{Q_i\cup Q_j}(f\tc \SG_{i,j})\tc\cG\big)\\
={}&(1+o(1))\rho_{i,j}\mu_{\bbT_n^2}\big(\var_{Q_i\cup Q_j}(f\tc \cS\cG_{i,j})\tc\cG\big).
\end{aligned}\label{eq:sw0}
     \end{multline}
Let $\cG(\L_{i,j})$ be the event that any $\ell_{N_\d}$ lattice sites contained in
$\L_{i,j}$ forming either a row or a column of some $Q_k$ contain an infection. With reference to Figure \ref{fig:obs:SG1}, we emphasise that the event $\cG(\L_{i,j})$ does not require anything about columns/rows which go out of $\L_{i,j}$. We define the event $\cG(\L_{i,j}^c)$ similarly with $\L_{i,j}$ replaced by
$\bbT_n^2\setminus \L_{i,j}$. Clearly $\cG\subset \cG(\L_{i,j})\cap \cG(\L_{i,j}^c)$ and
\begin{multline}
\label{eq:sw1}
\mu_{\bbT_n^2}\big(
    \var_{Q_i\cup Q_j}(f\tc \cS\cG_{i,j})\tc
      \cG\big)\\
\begin{aligned}[t]
\le {}&
      (1+o(1))\mu_{\bbT_n^2}\big(
    \var_{Q_i\cup Q_j}(f\tc \cS\cG_{i,j})\tc \cG(\L_{i,j})\cap \cG(\L_{i,j}^c)
      \big)\\={}& (1+o(1))\mu_{\bbT_n^2}\Big(\mu_{\L_{i,j}}\big(
    \var_{Q_i\cup Q_j}(f\tc \cS\cG_{i,j})\tc \cG(\L_{i,j})\big)\tc \cG(\L_{i,j}^c)
      \Big)
\end{aligned}
\end{multline}
In turn, the law of total variance implies that
\begin{gather}
\label{eq:sw2}  \mu_{\L_{i,j}}\big(
    \var_{Q_i\cup Q_j}(f\tc \cS\cG_{i,j})\tc \cG(\L_{i,j})\big)\le 
    \var_{\L_{i,j}}\big(f\tc \cS\cG_{i,j}\cap \cG(\L_{i,j})\big).
\end{gather}
Next comes a simple but key observation illustrated in Figure \ref{fig:obs:SG1}, whose formal proof is left to the reader.
\begin{obs}
  \label{obs:SG1}
The event $\SG_{i,j}\cap  \cG(\L_{i,j})$ implies the event $\SG(\L_{i,j})$.   
\end{obs}
Taking Observation \ref{obs:SG1} into account together with the inequality $\var(X\tc \cA)\le\frac{\bbP(\cB)}{\bbP(\cA)}\var(X\tc \cB)$,
valid for any random variable $X$ and events
$\cA\subset \cB$ with $\bbP(\cA)>0,$
we conclude that
\begin{equation}
    \var_{\L_{i,j}}\big(f\tc \cS\cG_{i,j}\cap \cG(\L_{i,j})\big)
\label{eq:sw3}     \le \frac{\m_{\L_{i,j}}(\SG(\L_{i,j}))}{\m_{\L_{i,j}}(\SG_{i,j}\cap\cG(\L_{i,j}))}\var_{\L_{i,j}}\big(f\tc \cS\cG(\L_{i,j})\big).
\end{equation}
From \eqref{eq:sw0}-\eqref{eq:sw3} we finally get
\begin{multline}
      \label{eq:sw4} 
\mu_{\bbT_n^2}\big(\1_{\SG_{i,j}}
    \var_{Q_i\cup Q_j}(f\tc \cS\cG_{i,j})\tc
      \cG\big)\\
\begin{aligned}[t]
\le{}& \frac{(1+o(1))\rho_{i,j}\m_{\L_{i,j}}(\SG(\L_{i,j}))}{\m_{\L_{i,j}}(\SG_{i,j}\cap\cG(\L_{i,j}))}\mu_{\bbT^2_n}\Big(\var_{\L_{i,j}}\big(f\tc \cS\cG(\L_{i,j})\big)\tc \cG(\L_{i,j}^c)\Big)\\
\le{}& (1+o(1))\m_{\L_{i,j}}(\SG(\L_{i,j}))\mu_{\bbT^2_n}\Big(\var_{\L_{i,j}}\big(f\tc \cS\cG(\L_{i,j})\big)\tc \cG(\L_{i,j}^c)\Big),     
\end{aligned}
\end{multline}
where we used 
\begin{align*}
\rho_{i,j}=\mu(\cS\cG_{i,j}\tc \cG
)\le{}& (1+o(1))\mu_{\L_{i,j}}(\SG_{i,j}\tc \cG(\L_{i,j}))\\
\le{}&(1+o(1))\m_{\L_{i,j}}(\SG_{i,j}\cap\cG(\L_{i,j}))
\end{align*}
to get the last inequality.
\begin{figure}
\centering
\begin{tikzpicture}[>=latex,x=1.0cm,y=1.0cm]
\draw [color=gray,dashed, xstep=1.0cm,ystep=1.0cm] (-1,-1) grid (3,3);
\fill[fill=black,fill opacity=0.25] (1,1) -- (2,1) -- (2,0) -- (1,0) -- cycle;
\draw (-0.73,2.74)-- (-0.73,-0.48);
\draw (-0.73,-0.48)-- (2.5,-0.48);
\draw (2.5,-0.48)-- (2.5,2.74);
\draw (2.5,2.74)-- (-0.73,2.74);
\draw[->>] (1,0.5)-- (0,0.5);
\draw[->>] (1.5,1)-- (1.5,2);
\draw[->>] (1.5,2)-- (1.5,2.74);
\draw[->>] (1.5,0)-- (1.5,-0.48);
\draw[->>] (0,0.5)-- (-0.73,0.5);
\draw[->>] (2,0.5)-- (2.5,0.5);
\draw(1.5,0.5) node {$Q_i$};
\draw (0,2) node {$\L_{i,j}$};
\end{tikzpicture}
\caption{\label{fig:obs:SG1}Illustration of Observation \ref{obs:SG1}. The shaded square of shape $\L^{(2N_\d)}$ is $\SG$ and the arrows indicate the presence of an infection in each row/column, as guaranteed by $\cG(\L_{i,j})$ with $\L_{i,j}$ being the larger square of shape $\L^{(2N)}$. Observation \ref{obs:SG1} asserts that these events combined imply $\SG(\L_{i,j})$ (see Figure \ref{fig:super-good}). The dashed lines delimit the boxes $Q_k$.}
\end{figure}
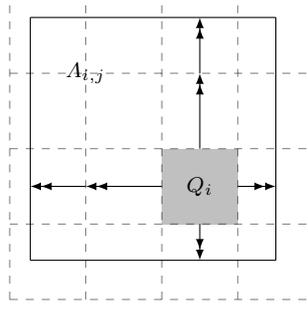

By applying Proposition \ref{thm:droplet} to the term $\var_{\L_{i,j}}(f \tc
\cS\cG(\L_{i,j}))$ and using that $c_x^{\L_{i,j}}\le c_x^{\bbT_n^2}$, we conclude that
\begin{multline*}
\mu_{\bbT_n^2}\big(\1_{\SG_{i,j}}
\var_{Q_i\cup Q_j}(f\tc \cS\cG_{i,j})\tc\cE\big)\\
\begin{aligned}\le{}&\begin{multlined}[t]
e^{O(\log^{3}(1/q)/\sqrt{q})}\m_{\L_{i,j}}(\SG(\L_{i,j}))\\
\times \sum_{x\in \L_{i,j}}\mu_{\bbT_n^2}\big(
      \mu_{\L_{i,j}}\big(c_x^{\L_{i,j}}\var_x(f)\tc
      \SG(\L_{i,j})\big)\tc \cG(\L_{i,j}^c)\big)
      \end{multlined}\\
 \le{}& e^{O(\log^{3}(1/q)/\sqrt{q})}\sum_{x\in \L_{i,j}}\mu_{\bbT_n^2}\big(\1_{\SG(\L_{i,j})}c_x^{\bbT_n^2}\var_x(f)\big),\end{aligned}
\end{multline*}
where we used $\m_{\bbT^2_n}(\cG(\L_{i,j}^c))=1-o(1)$ in the last inequality.
\end{proof}

\appendix
\section{Probability of super-good events}
\label{app:BP}
In this appendix we prove Proposition \ref{prop:probadroplet} and we gather several more technical and relatively standard bootstrap percolation estimates on the probability of super-good events used in Section \ref{sec:droplet}.

For $z>0$ we define
\[g(z)=-\log\big(\beta(1-e^{-z})\big),\]
where $\beta(u)=(u+\sqrt{u(4-3u)})/2$. It is known  \cite{Holroyd03}*{Proposition 5(ii)} that $\int_0^\infty g(z)\,dz=\pi^2/18$. We next recall some straightforward properties of $g$.
\begin{fact}
\label{fact:g}
The function $g$ is positive, decreasing, differentiable and convex on $(0,\infty)$. Moreover, the following asymptotic behaviour holds:
\begin{align*}
    g(z)\sim{}&\frac{1}{2}\log (1/z), &g'(z)\sim{}&\frac{-1}{2z},&\text{as }z\to{}&0,\\
    g(z)\sim{}&e^{-2z},&g'(z)\sim{}&-2e^{-2z},&\text{as }z\to{}&\infty,
\end{align*}
where $x\sim y$ stands for $x=(1+o(1))y$.
\end{fact}
The relevance of this function comes from its link to the probability of traversability.  Recalling Definition \ref{def:doubleg}, for any positive integers $a$ and $b$ we set
\begin{align*}
T^{\mathbf 1}(a,b)&{}=\m(\cT_{\rightarrow}^{{\mathbf 1}}(R(a,b))),&
T^{\mathbf 0}(a,b)&{}=\mu(\cT_{\rightarrow}^{\mathbf 0}(R(a,b)),
\end{align*}
where $\mathbf 0$ stands for the fully infected configuration. Note that these probabilities are the same for left-traversability, while for up or down-traversability $a$ and $b$ are inverted in the r.h.s. The next lemma follows easily from \cite{Holroyd03}*{Lemma 8}. Let $q'=-\log(1-q)=q+O(q^2)$.
\begin{lem}
\label{lem:T:Holroyd}
For any positive integers $a$ and $b$ and $\o\in\{\mathbf 0,\mathbf 1\}$ we have
\[T^{\o}(a,b)=q^{O(1)}e^{-ag(bq')}.\]
\end{lem}

\begin{cor}
\label{cor:T:ratio}
For any positive integers $a$ and $b$ we have
\begin{equation}
\label{eq:T:ratio}\max_{0\le s,s'\le a}\frac{T^{\mathbf 0}(s,b)T^{\mathbf 0}(a-s,b)}{T^{\mathbf 1}(s',b)T^{\mathbf 1}(a-s',b)}\le q^{-O(1)}.\end{equation}
Furthermore, for any boundary conditions $\o, \o'$ and rectangle $R$ of class $1\le n\le 2N$ (recall Definitions \ref{def:scales} and \ref{def:supergood}), we have
\begin{equation}
\label{eq:SG:position}\m_R(\SG^{\o}_s(R)\tc\SG^{\o'}(R))\ge q^{O(1)}
\end{equation}
uniformly over all possible values of $s$ and boundary conditions $\o,\o'$.
\end{cor}
\begin{proof}
Equation \eqref{eq:T:ratio} follows immediately from Lemma \ref{lem:T:Holroyd}. To obtain \eqref{eq:SG:position} with $n$ odd (the even case is treated identically), recall that 
\[\SG^{\o'}(R)=\bigcup_{s'}\SG^{\o'}_{s'}(R);\]
there are $q^{-O(1)}$ possible values of $s'$; by \eqref{eq:T:ratio},  for all $s$, $s'$, $\o$ and $\o'$, \[\m_R(\SG^{\o}_s)/\m_R(\SG^{\o'}_{s'})\ge q^{O(1)}.\qedhere\]
\end{proof}
We are now ready for the main result of this appendix. \begin{proof}[Proof of Proposition \ref{prop:probadroplet}]
We will prove the same bound for the super-good event occurring with all $s=0$ in Definition \ref{def:supergood} on all scales, \emph{i.e.}\ the initial infection $\L^{(0)}$ being in the bottom-left corner of $\L^{(n)}$. Once the offsets are fixed, it suffices to prove the bound on this probability for $n=2N$, in which case it reads
\begin{multline}
q\prod_{m=1}^{N}T^{\mathbf 1}(\ell_m-\ell_{m-1},\ell_{m})T^{\mathbf 1}(\ell_{m}-\ell_{m-1},\ell_{m-1})
\\=q^{O(N)}\exp\Big(-\sum_{m=1}^N (\ell_m-\ell_{m-1})(g(q'\ell_{m})+g(q'\ell_{m-1}))\Big),
\label{eq:SG:bound:auxiliary}
\end{multline}
by Lemma \ref{lem:T:Holroyd} and symmetry. Since $g$ is decreasing, the last sum is at most
\[
2\sum_{m=1}^{\infty} (\ell_m-\ell_{m-1})g(q'\ell_{m-1}).\]
The term for $m=1$ is $O(\log(1/q)/\sqrt{q})$ by Fact \ref{fact:g}. For the other terms we use that by convexity for any $0<a<b$
\[(b-a)g(a)\le\int_{a}^bg(z)\, dz - O((b-a)^2g'(a)).\]
Using Fact 
\ref{fact:g}, we get
\[-(b-a)^2g'(a)\le O((b-a))^2\times\begin{cases}
1/a&\text{if $a=O(1)$}\\
e^{-a} &\text{if $a=\Omega(1)$.}
\end{cases}\]

Finally, for $m\ge 2$ we have $\ell_m-\ell_{m-1}\le 2\sqrt q \ell_{m-1}$ by \eqref{def:fn}, so
\begin{align*}
q'\sum_{m=2}^{m_0}\frac{(\ell_m-\ell_{m-1})^2}{\ell_{m-1}}\le O(q'\sqrt q)\sum_{m=2}^{m_0}(\ell_m-\ell_{m-1})= O(q^{3/2}\ell_{m_0})={}&O(\sqrt{q})\\
(q')^{2}\sum_{m=m_0+1}^{\infty}(\ell_{m}-\ell_{m-1})^2e^{-q'\ell_{m-1}}\le O(q^3)\sum_{m=m_0+1}^{\infty}\ell_{m-1}^2e^{-q'\ell_{m-1}}={}&O(\sqrt q),
\end{align*}
setting $m_0=\max\{m,\ell_m\le 1/q\}$. Putting these bounds together and recalling \eqref{defn}, we obtain that the r.h.s.\ of \eqref{eq:SG:bound:auxiliary} is at least
\[\exp\Big(\frac{-2}{q'}\big(\int_{0}^\infty g(z)\, dz+O(\sqrt q)\log(1/q))\big)-\frac{O(\log^2(1/q))}{\sqrt q}\Big)=\exp\Big(-\frac{\pi^2}{9q}-\frac{O(\log^2(1/q))}{\sqrt q}\Big).\]
This concludes the proof of Proposition \ref{prop:probadroplet}.
\end{proof}

We next turn to the event $\overline\SG(V_2)$ from Definition \ref{def:SGbar} required in the proof of Lemma \ref{lem:lin}, so we fix $n=2m\in[2,2N)$.
\begin{lem}
 \label{lem:SGbar}
Recalling \eqref{overlineF1}, we have
\[\m_{\L^{(n)}}(\hat\cC_{1,2}\tc\SG)\ge q^{-O(1)}.\]
\end{lem}
\begin{proof}
Recall that $V_1\cup V_2=\L^{(n)}$ and assume $\SG(\L^{(n)})$ occurs. For any $0\le s_1,s_2\le \ell_{m}-\ell_{m-1}$ we write 
\[\SG_{s_1,s_2}(\L^{(n)})=\SG_{s_2}(\L^{(n)})\cap\SG_{s_1}(\L^{(n-1)}+s_2\vec e_2).\]
Then by Corollary \ref{cor:T:ratio} for any such $s_1,s_2$ we have
\[\mu_{\L^{(n)}}(\SG_{s_1,s_2}(\L^{(n)}))=\mu_{\L^{(n)}}(\SG(\L^{(n)}))q^{O(1)},\]
so it suffices to show that 
\[\m_{V_2}(\overline\SG_{0,0}(V_2))\ge \m_{\L^{(n)}}(\SG_{1,0}(\L^{(n)}))q^{O(1)},\]
since $\m(\cT_{\leftarrow}(I_1(\eta_{V_2})))\ge q$ for any $\eta_{V_2}\in\overline\SG(V_2)$.

However, by Definitions \ref{def:supergood} and \ref{def:SGbar} and symmetry we have
\begin{align*}\frac{\m_{V_2}(\overline\SG_{0,0}(V_2))}{\m_{\L^{(n)}}(\SG_{1,0}(\L^{(n)}))}={}&\frac{T^{\mathbf 1}(\ell_{m}-\ell_{m-1}-1,\ell_{m-1})T^{\mathbf 1}(\ell_m-\ell_{m-1},\ell_m-1)}{T^{\mathbf 1}(\ell_{m}-\ell_{m-1}-1,\ell_{m-1})T^{\mathbf 1}(\ell_m-\ell_{m-1},\ell_m)T^1(1,\ell_{m-1})}\\
\ge{}&\frac{T^{\mathbf 1}(\ell_m-\ell_{m-1},\ell_m-1)}{T^{\mathbf 1}(\ell_m-\ell_{m-1},\ell_m)}=q^{O(1)}e^{-(\ell_{m}-\ell_{m-1})(g((\ell_{m}-1)q')-g(\ell_mq'))},
\end{align*}
the last equality following from Lemma \ref{lem:T:Holroyd}.

By convexity of $g$ we get
\begin{equation}
\label{eq:g:increment}
g((\ell_m-1)q')-g(\ell_mq')\le -q'g'((\ell_m-1)q').
\end{equation}
By Fact \ref{fact:g} we have that the r.h.s.\ of \eqref{eq:g:increment}  is $O(1/\ell_m)$. Putting this together we obtain
\begin{equation}
\label{eq:muxiSG}\frac{\m_{V_2}(\overline\SG_{0,0}(V_2))}{\m_{\L^{(n)}}(\SG_{1,0}(\L^{(n)}))}\ge q^{O(1)}e^{-O(\ell_m-\ell_{m-1})/\ell_m}\ge q^{O(1)}e^{-O(\sqrt q)}=q^{O(1)},
\end{equation}
as desired, the second inequality coming from \eqref{def:fn} as in the proof of Proposition \ref{prop:probadroplet}.
\end{proof}

\section{Proof of Proposition \ref{prop:g-CBSEP}}\label{app:CBSEP}
Let $(\cS,\cS_1,\p)$ be the parameters of $g$-CBSEP on $\bbT_n^d$ and let $\ell=\lceil \p(\cS_1)^{-1/d}\rceil\ge 2$.  For simplicity we assume that $n^{1/d}/\ell \in \bbN$ and we partition the torus $\bbT_n^d$ into $M=(n/\ell)^d$ equal boxes $(B_j)_{j=1}^M$, where each $B_j$ is a suitable lattice translation by a vector in $\bbT_n^d$ of the box $B=[\ell]^d$. The labels of the boxes can be thought of as belonging to $\bbT_M^d$ and we say that $B_i,B_j$ are neighbouring boxes in $\bbT_n^d$ iff $i,j$ are nearest neighbours  in $\bbT_M^d$.

We then set $\hat
\cS= \cS^B, \hat \p((\s_x)_{x\in B})=\bigotimes_{x\in B}\p(\s_x),\hat
\cS_1=\bigcup_{x\in B}\{\s_x\in \cS_1\}$ and we consider the auxiliary renormalised $g$-CBSEP 
(in the sequel $\hat g$-CBSEP) on the graph
$\hat G= \bbT_M^d$ with parameters
$(\hat \cS, \hat \cS_1, \hat \p)$. Using the assumption $\lim_{n\to \infty}\p(\cS_1)=0$, we have that 
\[
\lim_{n\to \infty}\hat\p(\hat
\cS_1)=\lim_{n\to \infty}1-(1-\p(\cS_1))^{\ell^d}=1-e^{-1}.
\]
\begin{lem}
\label{lem:A1}Let  $\trel[\text{$\hat g$-CBSEP}]$ be the relaxation time of $\hat g$-CBSEP on $\widehat G$. Then there exists a constant $C=C(d)>0$ such that $\trel[\text{$\hat g$-CBSEP}]\le C$. 
\end{lem}
\begin{proof}
We closely follow \cite{Martinelli19a}*{Appendix A}.
Write $\hat \O_+$ for the space of $\hat g$-CBSEP configurations with at least one  particle and consider the projection $\varphi:\hat\O_+\mapsto \O_+$ given by $\varphi(\hat \o):=\{\1_{\{\hat\o_j\in \hat \cS_1\}}\}_{ j\in \hat G}.$ As discussed in Remark \ref{rem:CBSEP}, the projection of the $\hat g$-CBSEP chain is the CBSEP chain on $\hat G$ reversible w.r.t.\ $\p^+$, the product Bernoulli measure with parameter $p=\hat\p(\hat
\cS_1)$ conditioned on $\O_+$. 
For the latter, using $p=\Theta(1)$ as $n\to \infty$, it was proved in \cite{Hartarsky22CBSEP}*{Theorem 1} that its relaxation time $\trel[\text{CBSEP}]=O(1)$. Hence, it is enough to prove that $\trel[\text{$\hat g$-CBSEP}]\le C'\trel[\text{CBSEP}] $ for some constant $C'$.

Let $\hat\bbP_{\hat\o}(\cdot), \hat\bbE_{\hat\o}(\cdot)$ be the law and associated expectation of the $\hat g$-CBSEP chain with initial condition $\hat \o\in \hat \O_+$ and let $\bbP_{\eta}(\cdot), \bbE_{\eta}(\cdot)$ be the same objects for the projected chain (the CBSEP chain) with initial condition $\eta\in \O_+$. 

In order to prove the lemma, it is sufficient to prove that for any function $f:\hat\O_+\mapsto \bbR$ with zero mean w.r.t.\ $\hat\p^+$ and for any $\hat \o\in \hat \O_+$
the rate of exponential decay as $t\to +\infty$ of 
$|\hat\bbE_{\hat \o}(f(\hat\o(t)))|$ is at least $c/\trel[\text{CBSEP}]$ for some $c=c(p)>0$ independent of $f$ and $\hat \o$.

More formally, 
\[
\liminf_{t\to +\infty} -\frac 1t \log\big(|\hat\bbE_{\hat \o}(f(\hat\o(t)))|\big)\ge c/\trel[\text{CBSEP}]. 
\]
For any such $f$ write
\begin{align}\label{eq:B1}
& |\hat\bbE_{\hat \o}(f(\hat\o(t)))| \leq  \big|\hat\bbE_{\hat \o}\big(f(\hat\o(t)) \1_{\{\forall j\in\hat G,\,\t_j<t\}} \big) \big| 
+ \|f\|_\infty M\max_{j} \hat\bbP_{\hat \o}( \t_j\ge t), \end{align}
where $\t_j$ is  the first time such that $\varphi(\hat\o(t))_j\neq\varphi(\hat\o(0))$, which is measurable w.r.t.\ the projected chain.

It follows from standard tools for finite reversible Markov chains (see \emph{e.g.}\ \cite{Asselah01}*{Section 5} that 
there exists $K=K(\hat\o)<+\infty$ such that $\hat\bbP_{\hat \o}( \t_j\ge t)\le Ke^{-\l(j,\hat \o) t}
$
with
\[
\l(j,\hat \o)= \hat\pi^+\big(\hat\o':\ \varphi(\hat\o')_j\neq \varphi(\hat \o)_j\big) /\trel[\text{CBSEP}]\ge (p\wedge (1-p))/\trel[\text{CBSEP}].
\]
In particular, the rate of exponential decay as $t\to +\infty$ of the second term of the r.h.s.\ of \eqref{eq:B1} satisfies our requirement.

In order to prove a similar result for the first term in the r.h.s.\ of \eqref{eq:B1}, we observe that, conditionally on the event $\bigcap_j \{\t_j< t\}$ and on $\varphi(\hat\omega(t))$, the variables $(\hat\o_j(t))_{j\in \hat G}$ become independent with $\hat\o_j(t)\sim \hat \pi(\cdot |\varphi(\hat\omega(t))_j)$. Hence, if we set $g(\eta)=\hat\pi\big(f(\hat\o)|\varphi(\hat\o)=\eta\big),$ we get
\begin{align*}
\hat\bbE_{\hat \o}\big(f(\hat\o(t))\1_{\{\forall j\in\hat G,\,\t_j<t\}}\big)=
\bbE_{\varphi(\hat \o)}\big(g(\eta(t))\big) - \hat\bbE_{\varphi(\hat \o)}\big(g(\eta(t)) \1_{\{\exists j\in\hat G,\,\t_j\ge t\}}\big),
\end{align*}
so that
\begin{align*}
    \max_{\hat \o}\big|\hat\bbE_{\hat \o}\big(f(\hat\o(t))\1_{\{\forall j\in\hat G,\,\t_j<t\}}\big)\big|\le \max_\eta|\bbE_{\eta}\big(g(\eta(t))\big)| + \|f\|_\infty M\max_{j,\eta} \bbP_{\eta}( \t_j\ge t ).
\end{align*}
The rate of exponential decay as $t\to +\infty$ of both terms in the r.h.s.\ above is again at least $c/\trel[\text{CBSEP}]$ for some $c>0$, since $\pi^+(g)=\hat\pi^+(f)=0$.
\end{proof}

\begin{proof}[Proof of Proposition \ref{prop:g-CBSEP}]
For any pair of neighbouring boxes $B_i$ and $B_j$ we write $\hat \cE_{i,j}$ for the event $\bigcup_{x\in B_i\cup B_j}\{\s_x\in \cS_1\}$. Using Lemma \ref{lem:A1} and the definition of $\trel[\text{$\hat
  g$-CBSEP}]$ we get that
\begin{equation*}
\var_{\p^+_{\bbT_n^d}}(f)\le C \sum_{i\sim
    j}\p^+_{\bbT_n^d}\big(\1_{\hat\cE_{i,j}}\var_{B_i\cup B_j}(f\tc \hat\cE_{i,j})\big),
\end{equation*}
where the sum in the r.h.s.\ is an equivalent way to express the Dirichlet form of $\hat g$-CBSEP. Now fix a pair of adjacent boxes $B_i,B_j$ and let $\trel[\text{$g$-CBSEP}](i,j)$ be the relaxation time of our original $g$-CBSEP with parameters $(\cS,\cS_1,\p)$ on $B_i\cup B_j$. By symmetry $\trel[\text{$g$-CBSEP}](i,j)$ does not depend on $i,j$ and the common value will be denoted by $\trelt$. If we plug the Poincar\'e inequality for $g$-CBSEP on $B_i\cup B_j$
\begin{equation*}
\var_{B_i\cup B_j}(f\tc \hat\cE_{i,j})
  \le \trelt  \sum_{x\sim y\in B_i\cup B_j}\p_{B_i\cup B_j}^+\big(\1_{\cE_{x,y}}\var_{x,y}(f\tc \cE_{x,y})\big).
\end{equation*}
into the r.h.s.\ above, we get
\begin{align*}
\var_{\p^+_{\bbT_n^d}}(f)  \le{}& C\trelt \sum_{i\sim j}\sum_{x\sim y\in B_i\cup B_j}\p_{\bbT_n^d}^+\big(\1_{\hat\cE_{i,j}}\p_{B_i\cup B_j}^+\big(\1_{\cE_{x,y}}\var_{x,y}(f\tc \cE_{x,y})\big)\big)\\
 \le{}& 2dC\trelt \sum_{x\sim y \in \bbT_n^d}\p_{\bbT_n^d}^+\big(\1_{\cE_{x,y}}\var_{x,y}(f\tc \cE_{x,y})\big)\\
={}& 2dC\trelt\cD^{\text{$g$-CBSEP}}(f),
\end{align*}
where the second inequality uses $\1_{\hat\cE_{i,j}}\1_{\cE_{x,y}}=\1_{\cE_{x,y}}$ and \[\pi^+_{\bbT_n^d}\big(\1_{\hat\cE_{i,j}} \cdot\big)=\pi^+_{\bbT_n^d}(\hat\cE_{i,j})\pi_{\bbT_n^d\setminus(B_i\cup B_j)}\otimes \pi^+_{B_i\cup B_j}.\]
Thus, $\trel[\text{$g$-CBSEP}] \le O\big(\trelt\big)$.
It remains to bound $\trelt$ from above.

Let  
$\tmix[CBSEP]$ denote the mixing time of $g$-CBSEP on $B_{i}\cup B_{j}$ with
parameters $\cS'=\{0,1\},\cS'_1=\{1\}$ and $\pi'(1)=\p(\cS_1)=1-\pi'(0)$. Let $T_{\rm cov}^{\rm rw}$ be the cover time of the
continuous-time random walk on $B_{i}\cup B_{j}$. Theorem 2 of 
\cite{Hartarsky22CBSEP} implies $\trelt \le O(\tmix[CBSEP]+T_{\rm cov}^{\rm  rw})$. Moreover, it is well known (see \emph{e.g.}\ \cite{Levin09}) that $T_{\rm cov}^{\rm rw}$ is at most $O\big(\ell^d\log(\ell)\big)=O\big(\p(\cS_1)^{-1}\max(1,\log(1/\p(\cS_1)))\big)$
and \cite{Hartarsky22CBSEP}*{Corollary 3.1} proves\footnote{Strictly speaking
\cite{Hartarsky22CBSEP}*{Corollary 3.1} deals with the torus of cardinality $\p(\cS_1)^{-1}$ but
the same proof extends to our case of the box $B_{i}\cup B_{j}$.} the same bound for 
$\tmix[CBSEP]$. In conclusion,
\[\trelt \le
O\big(\p(\cS_1)^{-1}\max(1,\log(1/\p(\cS_1)))\big).\qedhere\]
\end{proof}
%
%

\begin{acknowledgements}
We acknowledge enlightening discussions with 
P.~Balister, B.~Bollob\'as, J.~Balogh, H.~Duminil-Copin, R.~Morris and P.~Smith
and the hospitality of IHES during the informal workshop in 2017 ``Kinetically constrained spin models and bootstrap percolation''.
On that occasion P. Balister suggested a flexible structure for the droplets, featuring freedom in the position of the internal core at all scales, which he conjectured would remove the spurious log-corrections in the bound \eqref{eq:FA2f:old} available at that time. We also thank the anonymous referees for the careful proofreading and  helpful comments on the presentation of the paper.
\end{acknowledgements}

\bibliographystyle{spmpsci}      
\bibliography{Bib}   

\begin{bibdiv}
\begin{biblist}

\bib{Aizenman88}{article}{
      author={Aizenman, M.},
      author={Lebowitz, J.~L.},
       title={Metastability effects in bootstrap percolation},
        date={1988},
        ISSN={0305-4470, 1361-6447},
     journal={J. Phys. A},
      volume={21},
      number={19},
       pages={3801\ndash 3813},
      review={\MR{968311}},
}

\bib{Arceri20}{article}{
      author={Arceri, F.},
      author={Landes, F.~P.},
      author={Berthier, L.},
      author={Biroli, G.},
       title={Glasses and aging: a statistical mechanics perspective},
        date={2020},
     journal={arXiv e-prints},
      eprint={2006.09725},
}

\bib{Asselah01}{article}{
      author={Asselah, A.},
      author={Dai~Pra, P.},
       title={Quasi-stationary measures for conservative dynamics in the
  infinite lattice},
        date={2001},
        ISSN={0091-1798},
     journal={Ann. Probab.},
      volume={29},
      number={4},
       pages={1733\ndash 1754},
      review={\MR{1880240}},
}

\bib{Balister16}{article}{
      author={Balister, P.},
      author={Bollob{\'a}s, B.},
      author={Przykucki, M.},
      author={Smith, P.},
       title={Subcritical {$\mathcal {U}$}-bootstrap percolation models have
  non-trivial phase transitions},
        date={2016},
        ISSN={0002-9947, 1088-6850},
     journal={Trans. Amer. Math. Soc.},
      volume={368},
      number={10},
       pages={7385\ndash 7411},
      review={\MR{3471095}},
}

\bib{Balogh12}{article}{
      author={Balogh, J.},
      author={Bollob{\'a}s, B.},
      author={{Duminil-Copin}, H.},
      author={Morris, R.},
       title={The sharp threshold for bootstrap percolation in all dimensions},
        date={2012},
        ISSN={0002-9947},
     journal={Trans. Amer. Math. Soc.},
      volume={364},
      number={5},
       pages={2667\ndash 2701},
      review={\MR{2888224}},
}

\bib{Balogh09a}{article}{
      author={Balogh, J.},
      author={Bollob{\'a}s, B.},
      author={Morris, R.},
       title={Bootstrap percolation in three dimensions},
        date={2009},
        ISSN={0091-1798},
     journal={Ann. Probab.},
      volume={37},
      number={4},
       pages={1329\ndash 1380},
      review={\MR{2546747}},
}

\bib{Berthier11}{article}{
      author={Berthier, L.},
      author={Biroli, G.},
       title={Theoretical perspective on the glass transition and amorphous
  materials},
        date={2011},
     journal={Rev. Mod. Phys.},
      volume={83},
      number={2},
       pages={587\ndash 645},
}

\bib{Blondel13}{article}{
      author={Blondel, O.},
      author={Cancrini, N.},
      author={Martinelli, F.},
      author={Roberto, C.},
      author={Toninelli, C.},
       title={Fredrickson-{{Andersen}} one spin facilitated model out of
  equilibrium},
        date={2013},
        ISSN={1024-2953},
     journal={Markov Process. Related Fields},
      volume={19},
      number={3},
       pages={383\ndash 406},
      review={\MR{3156958}},
}

\bib{Bollobas14}{article}{
      author={Bollob{\'a}s, B.},
      author={{Duminil-Copin}, H.},
      author={Morris, R.},
      author={Smith, P.},
       title={Universality of two-dimensional critical cellular automata},
        date={To appear},
     journal={Proc. Lond. Math. Soc.},
}

\bib{Bollobas15}{article}{
      author={Bollob{\'a}s, B.},
      author={Smith, P.},
      author={Uzzell, A.},
       title={Monotone cellular automata in a random environment},
        date={2015},
        ISSN={0963-5483},
     journal={Combin. Probab. Comput.},
      volume={24},
      number={4},
       pages={687\ndash 722},
      review={\MR{3350030}},
}

\bib{Bringmann12}{article}{
      author={Bringmann, K.},
      author={Mahlburg, K.},
       title={Improved bounds on metastability thresholds and probabilities for
  generalized bootstrap percolation},
        date={2012},
        ISSN={0002-9947},
     journal={Trans. Amer. Math. Soc.},
      volume={364},
      number={7},
       pages={3829\ndash 3859},
      review={\MR{2901236}},
}

\bib{Cancrini08}{article}{
      author={Cancrini, N.},
      author={Martinelli, F.},
      author={Roberto, C.},
      author={Toninelli, C.},
       title={Kinetically constrained spin models},
        date={2008},
        ISSN={0178-8051},
     journal={Probab. Theory Related Fields},
      volume={140},
      number={3-4},
       pages={459\ndash 504},
      review={\MR{2365481}},
}

\bib{Cancrini09}{incollection}{
      author={Cancrini, N.},
      author={Martinelli, F.},
      author={Roberto, C.},
      author={Toninelli, C.},
       title={Facilitated spin models: recent and new results},
        date={2009},
   booktitle={Methods of contemporary mathematical statistical physics},
      series={Lecture notes in math.},
      volume={1970},
   publisher={{Springer, Berlin, Heidelberg}},
       pages={307\ndash 340},
      review={\MR{2581609}},
}

\bib{Fredrickson84}{article}{
      author={Fredrickson, G.~H.},
      author={Andersen, H.~C.},
       title={Kinetic {{Ising}} model of the glass transition},
        date={1984},
     journal={Phys. Rev. Lett.},
      volume={53},
      number={13},
       pages={1244\ndash 1247},
}

\bib{Fredrickson85}{article}{
      author={Fredrickson, G.~H.},
      author={Andersen, H.~C.},
       title={Facilitated kinetic {{Ising}} models and the glass transition},
        date={1985},
     journal={J. Chem. Phys.},
      volume={83},
      number={11},
       pages={5822\ndash 5831},
}

\bib{Garrahan11}{incollection}{
      author={Garrahan, P.},
      author={Sollich, P.},
      author={Toninelli, C.},
       title={Kinetically constrained models},
        date={2011},
   booktitle={Dynamical heterogeneities in glasses, colloids and granular media
  and jamming transitions},
      editor={Berthier, L.},
      editor={Biroli, G.},
      editor={Bouchaud, J.-P.},
      editor={Cipelletti, L.},
      editor={{van Saarloos}, W.},
      series={International series of monographs on physics 150},
   publisher={{Oxford University Press, Oxford}},
       pages={341\ndash 369},
}

\bib{Graham97}{article}{
      author={Graham, I.~S.},
      author={Pich{\'e}, L.},
      author={Grant, M.},
       title={Model for dynamics of structural glasses},
        date={1997},
        ISSN={1063-651X, 1095-3787},
     journal={Phys. Rev. E},
      volume={55},
      number={3},
       pages={2132\ndash 2144},
}

\bib{Gravner08}{article}{
      author={Gravner, J.},
      author={Holroyd, A.~E.},
       title={Slow convergence in bootstrap percolation},
        date={2008},
        ISSN={1050-5164},
     journal={Ann. Appl. Probab.},
      volume={18},
      number={3},
       pages={909\ndash 928},
      review={\MR{2418233}},
}

\bib{Gravner12}{article}{
      author={Gravner, J.},
      author={Holroyd, A.~E.},
      author={Morris, R.},
       title={A sharper threshold for bootstrap percolation in two dimensions},
        date={2012},
        ISSN={0178-8051},
     journal={Probab. Theory Related Fields},
      volume={153},
      number={1-2},
       pages={1\ndash 23},
      review={\MR{2925568}},
}

\bib{Harris60}{article}{
      author={Harris, T.~E.},
       title={A lower bound for the critical probability in a certain
  percolation process},
        date={1960},
     journal={Math. Proc. Camb. Phil. Soc.},
      volume={56},
      number={1},
       pages={13\ndash 20},
      review={\MR{115221}},
}

\bib{Hartarsky21b}{article}{
      author={Hartarsky, I.},
       title={Bisection for kinetically constrained models revisited},
        date={2021},
        ISSN={1083-589X},
     journal={Electron. Commun. Probab.},
      volume={26},
       pages={Paper No. 60, 10},
      review={\MR{4346864}},
}

\bib{Hartarsky20II}{article}{
      author={Hartarsky, I.},
       title={Refined universality for critical {{KCM}}: upper bounds},
        date={2021},
     journal={arXiv e-prints},
      eprint={2104.02329},
}

\bib{Hartarsky22univlower}{article}{
      author={Hartarsky, I.},
      author={Mar{\^e}ch{\'e}, L.},
       title={Refined universality for critical {{KCM}}: lower bounds},
        date={2022},
     journal={Combin. Probab. Comput.},
      volume={31},
      number={5},
       pages={879\ndash 906},
      review={\MR{4472293}},
}

\bib{Hartarsky20}{article}{
      author={Hartarsky, I.},
      author={Mar{\^e}ch{\'e}, L.},
      author={Toninelli, C.},
       title={Universality for critical {{KCM}}: infinite number of stable
  directions},
        date={2020},
        ISSN={0178-8051, 1432-2064},
     journal={Probab. Theory Related Fields},
      volume={178},
      number={1},
       pages={289\ndash 326},
      review={\MR{4146539}},
}

\bib{Hartarsky21a}{article}{
      author={Hartarsky, I.},
      author={Martinelli, F.},
      author={Toninelli, C.},
       title={Universality for critical {{KCM}}: finite number of stable
  directions},
        date={2021},
     journal={Ann. Probab.},
      volume={49},
      number={5},
       pages={2141\ndash 2174},
      review={\MR{4317702}},
}

\bib{Hartarsky22CBSEP}{article}{
      author={Hartarsky, I.},
      author={Martinelli, F.},
      author={Toninelli, C.},
       title={Coalescing and branching simple symmetric exclusion process},
        date={2022},
     journal={Ann. Appl. Probab.},
      volume={32},
      number={4},
       pages={2841\ndash 2859},
      review={\MR{4474521}},
}

\bib{Hartarsky19}{article}{
      author={Hartarsky, I.},
      author={Morris, R.},
       title={The second term for two-neighbour bootstrap percolation in two
  dimensions},
        date={2019},
        ISSN={0002-9947},
     journal={Trans. Amer. Math. Soc.},
      volume={372},
      number={9},
       pages={6465\ndash 6505},
      review={\MR{4024528}},
}

\bib{Holroyd03}{article}{
      author={Holroyd, A.~E.},
       title={Sharp metastability threshold for two-dimensional bootstrap
  percolation},
        date={2003},
        ISSN={0178-8051},
     journal={Probab. Theory Related Fields},
      volume={125},
      number={2},
       pages={195\ndash 224},
      review={\MR{1961342}},
}

\bib{Levin09}{book}{
      author={Levin, D.~A.},
      author={Peres, Y.},
      author={Wilmer, E.~L.},
       title={Markov chains and mixing times},
   publisher={{American Mathematical Society, Providence, RI}},
        date={2009},
        ISBN={978-0-8218-4739-8},
        note={With a chapter by J. G. Propp and D. B. Wilson},
      review={\MR{2466937}},
}

\bib{Liggett05}{book}{
      author={Liggett, T.~M.},
       title={Interacting particle systems},
      series={Classics in mathematics},
   publisher={{Springer, Berlin, Heidelberg}},
        date={2005},
        ISBN={3-540-22617-6},
        note={Originally published by Springer, New York (1985)},
      review={\MR{2108619}},
}

\bib{Mareche20Duarte}{article}{
      author={Mar{\^e}ch{\'e}, L.},
      author={Martinelli, F.},
      author={Toninelli, C.},
       title={Exact asymptotics for {{Duarte}} and supercritical rooted
  kinetically constrained models},
        date={2020},
        ISSN={0091-1798},
     journal={Ann. Probab.},
      volume={48},
      number={1},
       pages={317\ndash 342},
      review={\MR{4079438}},
}

\bib{Martinelli19a}{article}{
      author={Martinelli, F.},
      author={Morris, R.},
      author={Toninelli, C.},
       title={Universality results for kinetically constrained spin models in
  two dimensions},
        date={2019},
        ISSN={0010-3616},
     journal={Comm. Math. Phys.},
      volume={369},
      number={2},
       pages={761\ndash 809},
      review={\MR{3962008}},
}

\bib{Martinelli19}{article}{
      author={Martinelli, F.},
      author={Toninelli, C.},
       title={Towards a universality picture for the relaxation to equilibrium
  of kinetically constrained models},
        date={2019},
        ISSN={0091-1798},
     journal={Ann. Probab.},
      volume={47},
      number={1},
       pages={324\ndash 361},
      review={\MR{3909971}},
}

\bib{Morris17a}{incollection}{
      author={Morris, R.},
       title={Monotone cellular automata},
        date={2017},
   booktitle={Surveys in combinatorics 2017},
      series={London math. {{Soc}}. {{Lecture}} note ser.},
      volume={440},
   publisher={{Cambridge University Press, Cambridge}},
       pages={312\ndash 371},
      review={\MR{3728111}},
}

\bib{Nakanishi86}{article}{
      author={Nakanishi, H.},
      author={Takano, H.},
       title={Numerical study on the kinetic {{Ising}} model for glass
  transition},
        date={1986},
        ISSN={03759601},
     journal={Phys. Lett. A},
      volume={115},
      number={3},
       pages={117\ndash 121},
}

\bib{Reiter91}{article}{
      author={Reiter, J.},
       title={Statics and dynamics of the two-spin\textendash facilitated
  kinetic {{Ising}} model},
        date={1991},
        ISSN={0021-9606, 1089-7690},
     journal={J. Chem. Phys.},
      volume={95},
      number={1},
       pages={544\ndash 554},
}

\bib{Saloff-Coste97}{incollection}{
      author={{Saloff-Coste}, L.},
       title={Lectures on finite {{Markov}} chains},
        date={1997},
   booktitle={Lectures on probability theory and statistics ({{Saint-Flour}},
  1996)},
      editor={Bernard, P.},
      series={Lecture notes in mathematics},
      volume={1665},
   publisher={{Springer, Berlin, Heidelberg}},
       pages={301\ndash 413},
      review={\MR{1490046}},
}

\bib{Shapira20}{article}{
      author={Shapira, A.},
       title={A note on the spectral gap of the
  {{Fredrickson}}\textendash{{Andersen}} one spin facilitated model},
        date={2020},
        ISSN={0022-4715, 1572-9613},
     journal={J. Stat. Phys.},
      volume={181},
      number={6},
       pages={2346\ndash 2352},
      review={\MR{4179809}},
}

\bib{Speck19}{article}{
      author={Speck, T.},
       title={Dynamic facilitation theory: a statistical mechanics approach to
  dynamic arrest},
        date={2019},
        ISSN={1742-5468},
     journal={J. Stat. Mech. Theory Exp.},
      number={8},
       pages={084015, 14},
      review={\MR{4031660}},
}

\bib{Teomy15}{article}{
      author={Teomy, E.},
      author={Shokef, Y.},
       title={Relation between structure of blocked clusters and relaxation
  dynamics in kinetically constrained models},
        date={2015},
        ISSN={1539-3755, 1550-2376},
     journal={Phys. Rev. E},
      volume={92},
      number={3},
       pages={032133, 10},
}

\bib{Toninelli05}{article}{
      author={Toninelli, C.},
      author={Biroli, G.},
      author={Fisher, D.~S.},
       title={Cooperative behavior of kinetically constrained lattice gas
  models of glassy dynamics},
        date={2005},
        ISSN={0022-4715},
     journal={J. Stat. Phys.},
      volume={120},
      number={1-2},
       pages={167\ndash 238},
      review={\MR{2165529}},
}

\bib{Uzzell19}{article}{
      author={Uzzell, A.~J.},
       title={An improved upper bound for bootstrap percolation in all
  dimensions},
        date={2019},
        ISSN={0963-5483},
     journal={Combin. Probab. Comput.},
      volume={28},
      number={6},
       pages={936\ndash 960},
      review={\MR{4015663}},
}

\end{biblist}
\end{bibdiv}

\end{document}